\documentclass[a4paper,twoside]{article}
\usepackage{a4}
\usepackage{amssymb}
\usepackage{amsmath}
\usepackage{bbm}
\usepackage{upref}
\usepackage[active]{srcltx}
\usepackage[pagebackref,colorlinks,citecolor=blue,linkcolor=blue]{hyperref}
\allowdisplaybreaks[2] 
\newcount\minutes \newcount\hours
\hours=\time
\divide\hours 60
\minutes=\hours
\multiply\minutes -60
\advance\minutes \time
\newcommand{\klockan}{\the\hours:{\ifnum\minutes<10 0\fi}\the\minutes}
\newcommand{\tid}{\today\ \klockan}
\newcommand{\prtid}{\smash{\raise 10mm \hbox{\LaTeX ed \tid}}}
\renewcommand{\prtid}{}
\makeatletter
\def\sectionmark#1{} 
\def\subsectionmark#1{}
\newcommand{\sectnr}{\ifnum \c@secnumdepth >\z@
                 \thesection.\hskip 1em\relax \fi}
\def\@evenhead{\footnotesize\rm\thepage\hfil\leftmark\hfil\llap{\prtid}}
\def\@oddhead{\footnotesize\rm\rlap{\prtid}\hfil\rightmark\hfil\thepage}
\def\tableofcontents{\section*{Contents} 
 \@starttoc{toc}}
\makeatother
\makeatletter
\def\@biblabel#1{#1.}
\makeatother
\makeatletter
\let\Thebibliography=\thebibliography
\renewcommand{\thebibliography}[1]{\def\@mkboth##1##2{}\Thebibliography{#1}
\addcontentsline{toc}{section}{References}
\frenchspacing 
\setlength{\@topsep}{0pt}
\setlength{\itemsep}{0pt}%
\setlength{\parskip}{0pt plus 2pt}%
}
\makeatother
\makeatletter
\def\mdots@{\mathinner.\nonscript\!.%
 \ifx\next,.\else\ifx\next;.\else\ifx\next..\else
 \nonscript\!\mathinner.\fi\fi\fi}
\let\ldots\mdots@
\let\cdots\mdots@
\let\dotso\mdots@
\let\dotsb\mdots@
\let\dotsm\mdots@
\let\dotsc\mdots@
\def\vdots{\vbox{\baselineskip2.8\p@ \lineskiplimit\z@
    \kern6\p@\hbox{.}\hbox{.}\hbox{.}\kern3\p@}}
\def\ddots{\mathinner{\mkern1mu\raise8.6\p@\vbox{\kern7\p@\hbox{.}}%
    \raise5.8\p@\hbox{.}\raise3\p@\hbox{.}\mkern1mu}}
\makeatother
\makeatletter
\let\Enumerate=\enumerate
\renewcommand{\enumerate}{\Enumerate%
\setlength{\@topsep}{0pt}
\setlength{\itemsep}{0pt}%
\setlength{\parskip}{0pt plus 1pt}%
\renewcommand{\theenumi}{\textup{(\alph{enumi})}}%
\renewcommand{\labelenumi}{\theenumi}%
}
\let\endEnumerate=\endenumerate
\renewcommand{\endenumerate}{\endEnumerate\unskip}
\makeatother
\makeatletter
\def\@seccntformat#1{\csname the#1\endcsname.\quad}
\makeatother
\newcommand{\authortitle}[2]{\author{#1}\title{#2}\markboth{#1}{#2}}
\newcommand{\auth}[2]{{#1, #2.}}
\newcommand{\art}[6]{{\sc #1, \rm #2, \it #3 \bf #4 \rm (#5), \mbox{#6}.}}

\newcommand{\artprep}[3]{{\sc #1, \rm #2, \rm #3.}}
\newcommand{\artin}[3]{{\sc #1, \rm #2,  in #3.}}
\newcommand{\book}[3]{{\sc #1, \it #2, \rm #3.}}
\newcommand{\AND}{{\rm and }}
\RequirePackage{amsthm}
\newtheoremstyle{descriptive}%
  {\topsep}   
  {\topsep}   
  {\rmfamily} 
  {}          
  {\bfseries} 
  {.}         
  { }         
  {}          
\newtheoremstyle{propositional}%
  {\topsep}   
  {\topsep}   
  {\itshape}  
  {}          
  {\bfseries} 
  {.}         
  { }         
  {}          
\theoremstyle{propositional}
\newtheorem{thm}{Theorem}[section]
\newtheorem{prop}[thm]{Proposition}
\newtheorem{lem}[thm]{Lemma}
\newtheorem{cor}[thm]{Corollary}
\theoremstyle{descriptive}
\newtheorem{deff}[thm]{Definition}
\newtheorem{example}[thm]{Example}
\newtheorem{remark}[thm]{Remark}
\makeatletter
\renewenvironment{proof}[1][\proofname]{\par
  \pushQED{\qed}%
  \normalfont 
  \trivlist
  \item[\hskip\labelsep
        \itshape
    #1\@addpunct{.}]\ignorespaces
}{%
  \popQED\endtrivlist\@endpefalse
}
\makeatother
\newcommand{\setm}{\setminus}
\renewcommand{\subsetneq}{\varsubsetneq}
\renewcommand{\emptyset}{\varnothing}
\def\vint{\mathop{\mathchoice%
          {\setbox0\hbox{$\displaystyle\intop$}\kern 0.22\wd0%
           \vcenter{\hrule width 0.6\wd0}\kern -0.82\wd0}%
          {\setbox0\hbox{$\textstyle\intop$}\kern 0.2\wd0%
           \vcenter{\hrule width 0.6\wd0}\kern -0.8\wd0}%
          {\setbox0\hbox{$\scriptstyle\intop$}\kern 0.2\wd0%
           \vcenter{\hrule width 0.6\wd0}\kern -0.8\wd0}%
          {\setbox0\hbox{$\scriptscriptstyle\intop$}\kern 0.2\wd0%
           \vcenter{\hrule width 0.6\wd0}\kern -0.8\wd0}}%
          \mathopen{}\int}
\newcommand{\Cp}{{C_p}}
\newcommand{\CpXeps}{C_p^{\clXeps}}
\newcommand{\CpXone}{C_p^{\clXone}}
\DeclareMathOperator{\diam}{diam}
\DeclareMathOperator{\capp}{cap}
\DeclareMathOperator{\Capp}{Cap}
\newcommand{\cp}{\capp_p}
\newcommand{\cpY}{\capp_p^Y}
\newcommand{\cpXeps}{\capp_p^{\clXeps}}
\DeclareMathOperator{\dist}{dist}
\DeclareMathOperator{\Euc}{Euc}
\DeclareMathOperator{\sgn}{sign}
\DeclareMathOperator{\inner}{in}
\DeclareMathOperator{\Lip}{Lip}
\DeclareMathOperator{\spt}{supp}
\newcommand{\supp}{\spt}
\DeclareMathOperator*{\osc}{osc}
\newcommand{\bdry}{\partial}
\newcommand{\bdy}{\bdry}
\newcommand{\loc}{_{\rm loc}}
\newcommand{\Lone}{\mathcal{L}^1}
\newcommand{\Ltwo}{\mathcal{L}^2}
\newcommand{\be}{\beta}
\newcommand{\simge}{\gtrsim}
\newcommand{\simle}{\lesssim}
\newcommand{\CPI}{C_{\rm PI}}
{\catcode`p =12 \catcode`t =12 \gdef\eeaa#1pt{#1}}      
\def\accentadjtext#1{\setbox0\hbox{$#1$}\kern   
                \expandafter\eeaa\the\fontdimen1\textfont1 \ht0 }
\def\accentadjscript#1{\setbox0\hbox{$#1$}\kern 
                \expandafter\eeaa\the\fontdimen1\scriptfont1 \ht0 }
\def\accentadjscriptscript#1{\setbox0\hbox{$#1$}\kern   
                \expandafter\eeaa\the\fontdimen1\scriptscriptfont1 \ht0 }
\def\accentadjtextback#1{\setbox0\hbox{$#1$}\kern       
                -\expandafter\eeaa\the\fontdimen1\textfont1 \ht0 }
\def\accentadjscriptback#1{\setbox0\hbox{$#1$}\kern     
                -\expandafter\eeaa\the\fontdimen1\scriptfont1 \ht0 }
\def\accentadjscriptscriptback#1{\setbox0\hbox{$#1$}\kern 
                -\expandafter\eeaa\the\fontdimen1\scriptscriptfont1 \ht0 }
\def\itoverline#1{{\mathsurround0pt\mathchoice
        {\rlap{$\accentadjtext{\displaystyle #1}
                \accentadjtext{\vrule height1.593pt}
                \overline{\phantom{\displaystyle #1}
                \accentadjtextback{\displaystyle #1}}$}{#1}}
        {\rlap{$\accentadjtext{\textstyle #1}
                \accentadjtext{\vrule height1.593pt}
                \overline{\phantom{\textstyle #1}
                \accentadjtextback{\textstyle #1}}$}{#1}}
        {\rlap{$\accentadjscript{\scriptstyle #1}
                \accentadjscript{\vrule height1.593pt}
                \overline{\phantom{\scriptstyle #1}
                \accentadjscriptback{\scriptstyle #1}}$}{#1}}
        {\rlap{$\accentadjscriptscript{\scriptscriptstyle #1}
                \accentadjscriptscript{\vrule height1.593pt}
                \overline{\phantom{\scriptscriptstyle #1}
                \accentadjscriptscriptback{\scriptscriptstyle #1}}$}{#1}}}}
\newcommand{\al}{\alpha}
\newcommand{\alp}{\alpha}
\newcommand{\ga}{\gamma}
\newcommand{\de}{\delta}
\newcommand{\eps}{\varepsilon}
\newcommand{\la}{\lambda}
\newcommand{\La}{\Lambda}
\newcommand{\sig}{\sigma}
\newcommand{\Om}{\Omega}
\newcommand{\gah}{\hat{\ga}}
\newcommand{\xhat}{\hat{x}}
\newcommand{\yhat}{\hat{y}}
\newcommand{\clOm}{{\overline{\Om}}}
\renewcommand{\phi}{\varphi}
\newcommand{\p}{{$p\mspace{1mu}$}}   
\newcommand{\xij}{x_{i,j}}
\newcommand{\Bij}{{B_{i,j}}}
\newcommand{\Bio}{{B_{i,0}}}
\newcommand{\Bxi}{{B_{x_i}}}
\newcommand{\Bxip}{{B'_{x_i}}}
\newcommand{\B}{{\cal B}}
\newcommand{\R}{\mathbf{R}}
\newcommand{\Sphere}{\mathbf{S}}
\newcommand{\clX}{\itoverline{X}}
\newcommand{\limplus}{{\mathchoice{\vcenter{\hbox{$\scriptstyle +$}}}
  {\vcenter{\hbox{$\scriptstyle +$}}}
  {\vcenter{\hbox{$\scriptscriptstyle +$}}}
  {\vcenter{\hbox{$\scriptscriptstyle +$}}}
}}
\newcommand{\limminus}{{\mathchoice{\vcenter{\hbox{$\scriptstyle -$}}}
  {\vcenter{\hbox{$\scriptstyle -$}}}
  {\vcenter{\hbox{$\scriptscriptstyle -$}}}
  {\vcenter{\hbox{$\scriptscriptstyle -$}}}
}}
\newcommand{\limpm}{{\mathchoice{\vcenter{\hbox{$\scriptstyle \pm$}}}
  {\vcenter{\hbox{$\scriptstyle \pm$}}}
  {\vcenter{\hbox{$\scriptscriptstyle \pm$}}}
  {\vcenter{\hbox{$\scriptscriptstyle \pm$}}}
}}
\newcommand{\Np}{N^{1,p}}
\newcommand{\Nploc}{N^{1,p}\loc}
\newcommand{\gat}{\tilde{\ga}}
\newcommand{\dt}{\tilde{d}}
\newcommand{\Omt}{\widetilde{\Om}}
\newcommand{\xt}{\tilde{x}}
\newcommand{\yt}{\tilde{y}}
\newcommand{\din}{d_{\inner}}
\newcommand{\Bin}{B_{\inner}}
\newcommand{\clXeps}{\itoverline{X}_\eps}
\newcommand{\clXone}{\itoverline{X}_1}
\newcommand{\mub}{\mu_\beta}
\newcommand{\Ga}{\Gamma}
\newcommand{\Lploc}{L^p\loc}
\newcommand{\CpY}{{C_p^Y}}
\newcommand{\zpm}{z_\limpm}
\newcommand{\zplus}{z_\limplus}
\newcommand{\zminus}{z_\limminus}
\newcommand{\ghat}{\hat{g}}
\numberwithin{equation}{section}
\newcommand{\imp}{\mathchoice{\quad \Longrightarrow \quad}{\Rightarrow}
                {\Rightarrow}{\Rightarrow}}
\newenvironment{ack}{\medskip{\it Acknowledgement.}}{}

\begin{document}

\authortitle{Anders Bj\"orn, Jana Bj\"orn
	and Nageswari Shanmugalingam}
            {Bounded geometry and \p-harmonic functions under uniformization
            and hyperbolization}
\author{
Anders Bj\"orn \\
\it\small Department of Mathematics, Link\"oping University, \\
\it\small SE-581 83 Link\"oping, Sweden\/{\rm ;}
\it \small anders.bjorn@liu.se
\\
\\
Jana Bj\"orn \\
\it\small Department of Mathematics, Link\"oping University, \\
\it\small SE-581 83 Link\"oping, Sweden\/{\rm ;}
\it \small jana.bjorn@liu.se
\\
\\
Nageswari Shanmugalingam 
\\
\it \small  Department of Mathematical Sciences, University of Cincinnati, \\
\it \small  P.O.\ Box 210025, Cincinnati, OH 45221-0025, U.S.A.\/{\rm ;}
\it \small shanmun@uc.edu
}

\date{}
\maketitle

\begin{center}
Dedicated to Pekka Koskela and the late Juha Heinonen on their 59th birthdays.
\end{center}

\noindent{\small
{\bf Abstract}. 
The uniformization and hyperbolization transformations formulated by
Bonk, Heinonen and Koskela in
\emph{``Uniformizing Gromov Hyperbolic Spaces''},
    Ast\'erisque {\bf 270} (2001), 
dealt with geometric properties of metric spaces. 
In this paper we
consider metric measure spaces  and construct a parallel
transformation of measures 
under the uniformization and hyperbolization procedures.
We show that if a locally compact roughly starlike Gromov hyperbolic space
is equipped with a measure that is uniformly locally doubling and
supports a uniformly local \p-Poincar\'e inequality, then
the transformed measure  is globally
doubling and supports a global \p-Poincar\'e inequality
on the corresponding uniformized space. 
In the opposite direction, we show that such global properties 
on bounded locally compact uniform spaces
yield similar uniformly local properties 
for the transformed measures on the corresponding hyperbolized spaces.

We use the above results on uniformization  of measures to
characterize when a Gromov hyperbolic space, equipped 
with a uniformly locally doubling measure supporting a
uniformly local \p-Poincar\'e
inequality, carries  nonconstant globally defined \p-harmonic functions 
with finite \p-energy.

We also study some geometric properties of Gromov hyperbolic and
uniform spaces.
While the
Cartesian product of two Gromov hyperbolic spaces need not be Gromov
hyperbolic, we construct an indirect product of 
such spaces that does result in a Gromov hyperbolic space.
This is done by first showing
that the Cartesian product of two bounded uniform domains is a uniform
domain. 

}

\bigskip
\noindent
    {\small \emph{Key words and phrases}:
      bounded geometry,
      doubling measure,
      finite-energy Liouville theorem,
  Gromov hyperbolic space,
  hyperbolization,
  metric space,
  \p-harmonic function,
  Poincar\'e inequality,
quasihyperbolic metric,
  uniform space,
  uniformization.
}


\section{Introduction}

Studies of metric space geometry usually consider two types of
synthetic (i.e.\ axiomatic) negative curvature conditions: Alexandrov
curvature (known as CAT($-1$) spaces) and Gromov hyperbolicity.  While
the Alexandrov condition governs both small and large scale
behavior of triangles,
the Gromov hyperbolicity
governs only the large scale behavior. 
As such, the Gromov hyperbolicity 
was eminently
suited to the study of hyperbolic groups, see e.g.\ 
Gromov~\cite{Gro}, Coornaert--Delzant--Papa\-do\-pou\-los~\cite{CDP} and
Ghys--de la Harpe~\cite{GdH}, while Bridson--Haefliger~\cite{BridHaef}
gives an excellent overview of both notions of curvature. 

Since the
ground-breaking work of Gromov~\cite{Gro}, the notion of Gromov
hyperbolicity has found applications in other parts of metric space
analysis as well.  
In~\cite[Theorem~1.1]{BHK-Unif}, Bonk, Heinonen and Koskela gave
a link between quasiisometry classes of locally compact roughly
starlike Gromov hyperbolic spaces and
quasisimilarity classes of locally compact bounded uniform spaces. 
In Buyalo--Schroeder~\cite{BuSch} it was shown that every complete
bounded doubling metric space is the visual boundary of a Gromov
hyperbolic space, see also Bonk--Saksman~\cite{BS}.

While none of the above mentioned studies, involving Gromov hyperbolic spaces
and uniform domains, considered how measures transform on such spaces
(see also e.g.\  Buckley--Herron--Xie~\cite{BHX} and
Herron--Shanmugalingam--Xie~\cite{HSX}),
analytic studies on metric spaces require measures as well. 
Although~\cite{BS} does consider function
spaces on certain Gromov hyperbolic spaces, called hyperbolic
fillings, these function spaces are associated with just the counting
measure on the vertices of such hyperbolic fillings and so do not lend
themselves 
to more general Gromov hyperbolic spaces.
Similar studies were undertaken in
Bonk--Saksman--Soto~\cite{BSS} and
Bj\"orn--Bj\"orn--Gill--Shanmugalingam~\cite{BBGS}.

In this paper we
seek to remedy this gap in the literature on analysis in Gromov
hyperbolic spaces.  Thus the primary focus of this paper
is to construct transformations of measures under the uniformization and
hyperbolization procedures, and to demonstrate how analytic
properties of the measure 
are preserved by them.
The analytic properties of
interest here are the doubling property and the Poincar\'e inequality,
assumed either globally on the uniform spaces, or uniformly locally
(i.e.\ for balls up to some fixed radius) on
the Gromov hyperbolic spaces.
As trees are the quintessential
models of Gromov hyperbolic spaces, the results in this paper are
motivated in part by the results in~\cite{BBGS}.

The following is our main result,
combining Theorems~\ref{thm-comp-largest-Whitney}
and~\ref{thm-PI-Xeps}.
Here, $z_0\in X$ is a fixed uniformization center and $\eps_0(\de)$ is
as in Bonk--Heinonen--Koskela~\cite{BHK-Unif}, see
later sections for relevant definitions.

\begin{thm} \label{thm-main}
Assume that
$(X,d)$ is a locally compact roughly starlike Gromov $\de$-hyperbolic space
  equipped with a measure
$\mu$ which is doubling on $X$ for balls of radii at most $R_0$,
with a doubling constant $C_d$.
Let $X_\eps=(X,d_\eps)$ be the uniformization of $X$ given for
  $0<\eps\le\eps_0(\de)$ by
\[
d_\eps(x,y) = \inf_\ga \int_\ga e^{-\eps d(\cdot,z_0)}\,ds,
\]
with the infimum taken over all
rectifiable curves $\ga$ in $X$ joining $x$ to $y$.
Also let
\[
\be > \frac{17 \log C_d}{3R_0} \quad \text{and} \quad
d\mu_\be = e^{-\be d(\cdot,z_0)}\,d\mu.
\]
Then the following are true\/\textup{:}
\begin{enumerate}
\item \label{i-a}
$\mu_\be$ is globally doubling both on $X_\eps$ and its completion $\clXeps$.
\item \label{i-b}
If $\mu$ supports a  \p-Poincar\'e inequality for balls of radii at most $R_0$,
then $\mu_\be$ supports a global \p-Poincar\'e inequality
both on $X_\eps$ and $\clXeps$.
\end{enumerate} 
\end{thm}

Along the way, we also show that if the assumptions hold with some
value of $R_0$ then they  hold for any value of $R_0$
at the cost of enlarging $C_d$,
  see Proposition~\ref{prop-doubling-R0-quasiconvex}
  and Theorem~\ref{thm-PI-arb-R0}.

We also obtain the following corresponding result for the
hyperbolization procedure, see
Propositions~\ref{prop:doubling-mu-alph} and~\ref{prop:PI-mu-alpha}.

\begin{thm}\label{thm-main2}
Let $(\Om,d)$ be a
locally compact bounded uniform space, equipped with
  a globally doubling measure $\mu$.
Let $k$ be the quasihyperbolic metric on $\Om$, given by
\begin{equation} \label{eq-qhyp-metric}
  k(x,y)=\inf_\gamma\int_\gamma \frac{ds}{d_\Om(\,\cdot\,)},
\end{equation}
where  $d_\Om(x)=\dist(x,\bdy \Om)$
and the infimum is taken
over all rectifiable curves $\gamma$ in $\Om$ connecting $x$ to $y$.
For $\alpha>0$ we equip
the corresponding Gromov hyperbolic space
$(\Om,k)$ with the measure
$\mu^\alpha$ given by $d\mu^\alpha=d_\Om(\,\cdot\,)^{-\alpha}\,d\mu$.
Let $R_0 >0$.

Then the following are true\/\textup{:}
\begin{enumerate}
\item \label{ii-a}
$\mu^\alpha$ is doubling on $(\Om,k)$ for balls of radii at most  $R_0$. 
\item \label{ii-b}
  If $\mu$ supports a  global \p-Poincar\'e inequality, then
$\mu^\alpha$ supports a \p-Poincar\'e inequality for balls of radii at most $R_0$.
\end{enumerate} 
\end{thm}

We use Theorem~\ref{thm-main}  to study potential
theory on locally compact roughly starlike
Gromov hyperbolic spaces, equipped with a locally uniformly doubling measure
supporting a uniformly local Poincar\'e inequality.
In particular, we characterize when the finite-energy
Liouville theorem  holds on such spaces, i.e.\
when there exist no nonconstant globally defined
\p-harmonic functions with finite \p-energy. 
The characterization is given in terms of
the nonexistence of two disjoint compact
sets of positive \p-capacity in the
boundary 
of the uniformized space, see Theorem~\ref{thm-nonconst-energy-iff-pos-cap}.
This characterization complements
our results in \cite{BBSliouville}.

As already mentioned, an in-depth study of locally compact roughly starlike Gromov
hyperbolic spaces, as well as
links between them and bounded locally compact uniform domains,
was undertaken in the seminal work Bonk--Heinonen--Koskela~\cite{BHK-Unif}.
They showed \cite[the discussion before Proposition~4.5]{BHK-Unif}
that the operations of uniformization and hyperbolization are mutually
opposite:
\begin{itemize}
\item 
  A uniformization followed by a hyperbolization takes a given
  locally compact roughly starlike Gromov
hyperbolic  space $X$ to a roughly
starlike Gromov hyperbolic  space which is biLipschitz equivalent to $X$,
  see \cite[Proposition~4.37]{BHK-Unif}.
  (Note that in \cite{BHK-Unif} ``quasiisometric'' means biLipschitz.)
\item
A hyperbolization of a bounded locally compact uniform space $\Om$, followed by a
uniformization, returns a bounded uniform space which is quasisimilar
to $\Om$, see \cite[Proposition~4.28]{BHK-Unif}.
\end{itemize}

Here, a homeomorphism
$\Phi:X\to Y$ between two noncomplete metric spaces 
is \emph{quasisimilar} if it is
$C_x$-biLipschitz on every ball $B(x,c_0 \dist(x,\partial X))$, for some
$0<c_0<1$ independent of $x$, and there exists a
homeomorphism $\eta:[0,\infty)\to[0,\infty)$ 
such that for each
distinct triple of points $x,y,z\in X$, 
\[
\frac{d_Y(\Phi(x),\Phi(y))}{d_Y(\Phi(x),\Phi(z))} \le
\eta\biggl(\frac{d_X(x,y)}{d_X(x,z)}\biggr).
\]

It was also shown in~\cite[Theorem~4.36]{BHK-Unif} that two roughly
starlike Gromov hyperbolic spaces are biLipschitz equivalent if and
only if any two of their uniformizations are quasisimilar.

We continue the study of Gromov hyperbolic spaces in this spirit by
considering pairs of Gromov hyperbolic spaces in
Section~\ref{unif-gromov-skewproduct}. 
Note that the Cartesian product of two Gromov hyperbolic spaces need
not be Gromov hyperbolic, as demonstrated by $\R\times\R$, which is
\emph{not} a Gromov hyperbolic space even though $\R$ is.
On the other
hand, in Section~\ref{unif-gromov-skewproduct} we obtain the following
result, see Proposition~\ref{prop-unif-product} for a more
  precise result.

\begin{prop}
Let $(\Om,d)$ and $(\Om',d')$ be two bounded uniform spaces.
Then $\Om \times\Om'$ is a bounded uniform space with respect
to the metric
\[
\dt((x,x'),(y,y')) = d(x,y) + d'(x',y').
\]
\end{prop}

We use this, together with the results of~\cite{BHK-Unif}, to construct an
indirect product $X\times_\eps Y$ of two Gromov hyperbolic spaces which 
is also Gromov hyperbolic, see Section~\ref{unif-gromov-skewproduct}. 
In  this section we also study
properties of such product hyperbolic spaces.
For a fixed Gromov
$\delta$-hyperbolic space $X$, there is a whole family of
uniformizations $X_\eps$, one for each $0<\eps\le \eps_0(\delta)$.
As mentioned above, $X_{\eps}$ is quasisimilar to $X_{\eps'}$ when
$0<\eps,\eps'\le \eps_0(\delta)$.

On the other hand, we show in Proposition~\ref{prop-canonical-Lip}
that the canonical identity mapping between
$X\times_\eps Y$ and $X\times_{\eps'} Y$ is never biLipschitz
if $\eps\ne\eps'$, and it is even possible that the two indirect
products 
are not even quasiisometric.
Here, a map $\Phi:Z\to W$
is a \emph{quasiisometry} (also called, perhaps more accurately,
 rough quasiisometry as in \cite{BHK-Unif}
and \cite{BBGS}) if there are $C>0$ and $L\ge 1$
such that the $C$-neighborhood of $\Phi(Z)$ contains $W$ and for all $z,z'\in Z$,
\[
\frac{d(z,z')}{L}-C \le d(\Phi(z),\Phi(z'))\le Ld(z,z')+C.
\]
It is not difficult to show that visual boundaries of quasiisometric
locally compact roughly starlike Gromov hyperbolic spaces are quasisymmetric,
see e.g.\ Bridson--Haefliger~\cite[Theorem~3.22]{BridHaef}.
We take advantage of this to show the quasiisometric nonequivalence
of two indirect products
of the hyperbolic disk and $\R$, see Example~\ref{ex:non-qs}.

The broad organization of the paper is as follows.
Background definitions and preliminary results are
given in Sections~\ref{sect-Gromov} and~\ref{sect-doubling}, while the
definition of Poincar\'e inequalities is
given in Section~\ref{sect-defn-of-PI-ug}.  
The main aims in
Sections~\ref{sect-doubling-of-mubeta} and~\ref{sect-PI-for-mubeta}
are to deduce parts~\ref{i-a} and \ref{i-b}, respectively,
of Theorem~\ref{thm-main}.
The dual transformation of hyperbolization,
via the quasihyperbolic metric \eqref{eq-qhyp-metric}, is discussed in
Section~\ref{sect-hyperbolization}, where
also Theorem~\ref{thm-main2} is shown.
The above sections fulfill the
main goal of this paper, and form a basis for comparing the potential
theories on Gromov hyperbolic spaces 
and on uniform spaces.

The remaining sections are devoted to applications of the results
obtained in the preceding 
sections.
In Section~\ref{unif-gromov-skewproduct} we construct and study
the indirect product, 
providing a family of new Gromov hyperbolic
spaces from a pair of Gromov hyperbolic spaces.
The subsequent sections are devoted to the
impact of uniformization and hyperbolization procedures on nonlinear
potential theory.
In Section~\ref{sect-N1p-p-harm} we discuss
Newton--Sobolev spaces and \p-harmonic functions, and
then in Section~\ref{p-harm-appl} we show that under certain natural
conditions, the class of \p-harmonic functions is preserved under the
uniformization and hyperbolization procedures.
In this final section, we also characterize which Gromov hyperbolic spaces
with bounded geometry support the finite-energy Liouville theorem
for \p-harmonic functions.

In the beginning of each section,
we list the
standing assumptions for that section in italicized text;
in Sections~\ref{sect-Gromov} and~\ref{sect-doubling-of-mubeta}
these assumptions are given a little later.

\begin{ack}
The discussions leading to this paper started in 2013, while
the authors were visiting Institut Mittag-Leffler.
They continued during the parts of 
2017 and 2018 when N.~S. was a guest 
professor at Link\"oping University,
partially
funded by the Knut and Alice Wallenberg Foundation,
and during the parts of 2019 when A.~B. and J.~B. were Taft Scholars
at the University of Cincinnati.
The authors would like to thank these institutions for 
their kind support and hospitality.

A.~B. and J.~B. were partially supported by the Swedish Research Council
grants 2016-03424 and 621-2014-3974, respectively. 
N.~S. was partially supported by the NSF grants DMS-1500440
and DMS-1800161.

The authors would also like to thank Xining Li for a fruitful
discussion on annular quasiconvexity
leading to  Lemma~\ref{lem-annular}.
\end{ack}

\section{Gromov hyperbolic spaces}
\label{sect-Gromov}

A \emph{curve} is a continuous mapping from an interval.
Unless stated otherwise,
we will only consider curves which are defined on compact intervals.
We denote the length of a curve $\ga$  by $l_\ga=l(\ga)$,
and a curve is \emph{rectifiable} if it has finite length.
Rectifiable curves can be parametrized by arc length $ds$.

A metric space $X=(X,d)$ is \emph{$L$-quasiconvex} if 
for each $x,y\in X$ there is a curve $\gamma$ with end points $x$ and
$y$ and length $l_\gamma\le L d(x,y)$.
$X$ is a \emph{geodesic} space if it is $1$-quasiconvex,
    and $\ga$ is then a \emph{geodesic}
from $x$ to $y$.
We will consider a related metric, called
the  \emph{inner metric}, given by
\begin{equation} \label{eq-din}
\din(x,y):=\inf_\ga l_\ga 
\quad \text{for all } x,y \in X,
\end{equation}
where the infimum is taken over all curves $\ga$
from $x$ to $y$.
If $(X,d)$ is quasiconvex, then $d$ and $\din$ are
biLipschitz equivalent metrics on $X$. The space
$X$ is a \emph{length space} if
$\din(x,y)=d(x,y)$ for all $x,y\in X$.
  By Lemma~4.43 in \cite{BBbook}, arc length is the same
  with respect to $d$ and $\din$, 
and thus $(X,\din)$ is a length space.
A metric space is \emph{proper} if all closed bounded sets are compact.
A proper length space
is necessarily a geodesic space,
  by  Ascoli's theorem or the Hopf--Rinow theorem below.
 To avoid pathological situations, all metric spaces in
    this paper are assumed to contain at least two points.

We denote balls in $X$ by $B(x,r)=\{y \in X: d(y,x) <r\}$ and the
scaled concentric ball by
$\la B(x,r)=B(x,\la r)$. 
In metric spaces it can happen that
balls with different centers and/or radii denote the same set. 
We will however adopt the convention that a ball comes with
a predetermined center and radius. 
Similarly, when we say that $x \in \ga$ we mean that
  $x=\ga(t)$ for some $t$.
  If $\ga$ is noninjective, this $t$ may not be unique, but we are
always implicitly referring to
  a  specific such $t$.

\begin{thm} \label{thm-Hopf-Rinow}
  \textup{(Hopf--Rinow theorem)}
  If $X$ is a complete locally compact length  space,
  then it is proper and geodesic.
\end{thm}

This version is a generalization of the original theorem,
see e.g.\ Gromov~\cite[p.~9]{GromovBook} for a proof.

\begin{deff}
A complete unbounded geodesic metric space $X$ is
\emph{Gromov hyperbolic}
if
there is  a \emph{hyperbolicity constant} $\delta\ge0$ such that
whenever 
$[x,y]$, $[y,z]$ and $[z,x]$ are geodesics in $X$,
every point $w\in[x,y]$ lies within a distance
$\delta$ of $[y,z]\cup[z,x]$.
\end{deff}

The ideal Gromov hyperbolic space is a metric tree, which 
is Gromov hyperbolic with $\de=0$.
A \emph{metric tree} is a tree where each edge is considered  to be a geodesic of unit length.

\begin{deff}
An unbounded metric space  $X$ is \emph{roughly starlike}
if there are some
$z_0\in X$ and $M>0$ such that whenever $x\in X$ there is a geodesic ray
$\gamma$ in $X$, starting from $z_0$, such that $\dist(x,\gamma)\le M$.
A \emph{geodesic ray} is a curve $\ga:[0,\infty) \to X$ with infinite length
  such that $\ga|_{[0,t]}$ is a geodesic for each $t > 0$.
\end{deff}
  
If $X$ is a roughly starlike Gromov hyperbolic space, then
the roughly starlike condition holds for every choice of $z_0$,
although $M$ may change.

\begin{deff}
A nonempty open set $\Om\subsetneq X$ in a metric space $X$
is an \emph{$A$-uniform domain}, with $A\ge1$,  
if for every pair $x,y\in\Om$
there is a rectifiable arc length parametrized
curve $\ga: [0,l_\ga] \to \Om$ with $\ga(0)=x$ and
$\ga(l_\ga)=y$ such that
$l_\ga \le A d(x,y)$ and
\[
d_\Om(\ga(t)) \ge \frac{1}{A} \min\{t, l_\ga-t\}
    \quad \text{for } 0 \le t \le l_\ga,
\]
where
\[ 
  d_\Om(z)=\dist(z,X \setm \Om),
  \quad z \in \Om.
\] 
The curve $\ga$ is said to be an \emph{$A$-uniform curve}.
A noncomplete
metric space $(\Om,d)$ is \emph{$A$-uniform} if it is an $A$-uniform
domain in its completion.

A ball $B(x,r)$ in a uniform space $\Om$
  is a \emph{subWhitney ball} if
  $r \le c_0 d_\Om(x)$, where $0<c_0<1$ is a predetermined constant.
We will primarily use $c_0=\frac12$.
\end{deff}

The completion of a locally compact uniform space is always proper,
by Proposition~2.20 in
Bonk--Heinonen--Koskela~\cite{BHK-Unif}.
Unlike the definition used in \cite{BHK-Unif},
we do not require uniform spaces to be locally compact.

It follows directly from the definition that an $A$-uniform space is
$A$-quasiconvex.
One might ask if the uniformity assumption in 
Proposition~2.20 in~\cite{BHK-Unif} can be replaced by a quasiconvexity
assumption, i.e.\
if the completion of a locally compact quasiconvex
space is always proper, however the following
example shows that this can fail even if the original space is geodesic.
Thus the uniformity assumption in
Proposition~2.20 in~\cite{BHK-Unif} is really crucial.

\begin{example}
Let
\[
X=\biggl\{\{x_j\}_{j=1}^\infty :
\sum_{j=1}^\infty |x_j|  \le 1, \ 0 < x_1 \le 1, \text{ and } x_n=0 \text{ if } x_1>\frac{1}{n},
\ n=2,3,\ldots \biggr\},
\]
equipped with the $\ell^1$-metric.
Then $X$ is a bounded locally compact geodesic space which is
not totally bounded, and thus has a nonproper completion.
\end{example}

\medskip

\emph{We assume from now on that $X$ is a
  locally compact roughly starlike Gromov $\de$-hyperbolic space.
  We also fix a point $z_0 \in X$ and let
$M$ be the constant in the roughly starlike condition with respect to
$z_0$.}

\medskip

By the Hopf--Rinow Theorem~\ref{thm-Hopf-Rinow},  $X$ is proper.
The point $z_0$ will serve as a center for the uniformization $X_\eps$ of $X$.
Following Bonk--Heinonen--Koskela~\cite{BHK-Unif},
we define
\[
(x|y)_{z_0}:=\tfrac12[d(x,z_0)+d(y,z_0)-d(x,y)],
\quad x,y \in X,
\]
and, for a fixed $\eps>0$,
 the \emph{uniformized metric} $d_\eps$ on $X$
as
\[ 
  d_\eps(x,y) = \inf_\ga \int_\ga \rho_\eps\,ds,
  \quad \text{where } \rho_\eps(x)=e^{-\eps d(x,z_0)}
\] 
and the infimum is taken over all
rectifiable curves $\ga$ in $X$ joining $x$ to $y$.
Note that if $\ga$ is a compact curve in $X$,
then  $\rho_\eps$ is bounded from above and away
from $0$ on $\ga$, and in particular $\ga$ is rectifiable with respect to $d_\eps$
if and only if it is rectifiable with respect to $d$.

The set $X$, equipped with the metric $d_\eps$, is denoted by $X_\eps$.
We let $\clXeps$ be the completion of $X_\eps$, and let
$\partial_\eps X = \clXeps\setminus X_\eps$.
When writing e.g.\ $B_\eps$, $\diam_\eps$ and $\dist_\eps$
the $\eps$ indicates that these notions are taken with respect to
$(\clXeps,d_\eps)$.
The length of the curve $\ga$ with respect to $d_\eps$
is denoted by $l_\eps(\ga)$,
and arc length $ds_\eps$ with respect to $d_\eps$ satisfies
\[ 
  ds_\eps = \rho_\eps\,ds.
\] 
It follows that $X_\eps$ is a length space, and thus also $\clXeps$
is a length space.
By a direct calculation (or \cite[(4.3)]{BHK-Unif}),
$\diam_\eps \clXeps= \diam_\eps X_\eps  \le 2/\eps$.
Note that as a set, $\partial_\eps X$ is
independent of $\eps$ and depends only on the Gromov hyperbolic
structure of $X$, see e.g.\ 
\cite[Section~3]{BHK-Unif}.
The notation adopted in~\cite{BHK-Unif} is $\partial_G X$.

The following important theorem is due to Bonk--Heinonen--Koskela~\cite{BHK-Unif}.

\begin{thm} \label{thm-eps0}
  There is a constant $\eps_0=\eps_0(\de)>0$ only depending on $\de$ such that
  if $0 < \eps \le \eps_0(\de)$, then
  $X_\eps$ is 
  an $A$-uniform space for some $A$ depending only on $\de$,
  and $\clXeps$ is a compact geodesic space.

  If $\de=0$, then $\eps_0(0)$ can be chosen arbitrarily large.
\end{thm}

In the proof below we recall the relevant references from
\cite{BHK-Unif} and specify the dependence on $\delta$.

\begin{proof}
  By Proposition~4.5 in  \cite{BHK-Unif} there is $\eps_0(\de)>0$ such that if
  $0< \eps \le \eps_0(\de)$, then $X_\eps$ is 
an $A$-uniform space for some $A$ depending only on $\de$.
As $X_\eps$ is bounded, it follows from  Proposition~2.20 in \cite{BHK-Unif}
that  $\clXeps$ is a compact length space,
which by
Ascoli's theorem or
the Hopf--Rinow Theorem~\ref{thm-Hopf-Rinow} is geodesic.

The bound $\eps_0(\de)$ in Proposition~4.5 in \cite{BHK-Unif} is only needed for
the Gehring--Hayman lemma to be true, see \cite[Theorem~5.1]{BHK-Unif}.
If $\de=0$, then any curve from $x$ to $y$ contains  the unique geodesic $[x,y]$ as a subcurve.
From this the Gehring--Hayman lemma follows directly without any bound on $\eps$.
Note that in this case it also follows that a curve in
$X=X_\eps$ is simultaneously a geodesic
with respect to $d$ and $d_\eps$.
\end{proof}  

We recall, for further reference, the following key estimates
from \cite{BHK-Unif}.

\begin{lem}   \label{lem-BHK-dist-rho}
{\rm (\cite[Lemma~4.10]{BHK-Unif})}
There is a constant $C(\de)\ge1$
such that for all $0<\eps\le\eps_0=\eps_0(\de)$ and all $x,y\in X$,
\begin{equation}   \label{eq-BHK-dist}
\frac{1}{C(\de)} d_\eps(x,y) 
\le \frac{\exp(-\eps(x|y)_{z_0})}{\eps} \min\{1,\eps d(x,y)\} 
\le C(\de) d_\eps(x,y).
\end{equation}
\end{lem}

\begin{lem}\label{lem:dist-to-eps-bdry}
\textup{(\cite[Lemma~4.16]{BHK-Unif})}
Let $\eps>0$. If $x\in X$, then
\begin{equation}\label{eq-BHK-d-rho}
  \frac{e^{-\eps d(x,z_0)}}{e\eps}\le \dist_\eps(x,\partial_\eps X)=:d_\eps(x)
   \le C_0 \, \frac{e^{-\eps d(x,z_0)}}{\eps},
\end{equation}
where $C_0=2e^{\eps M}-1$. In particular, $\eps d_\eps(x) \simeq \rho_\eps(x)$, and
$x\to\bdy_\eps X$ with respect to $d_\eps$ if and only if $d(x,z_0)\to\infty$.
\end{lem}

Note that one may
choose
$C_0 = 2e^{\eps_0     M}-1$
for it to  be independent of $\eps$, provided that $0<\eps\le \eps_0$.

\begin{cor}  \label{cor-comp-d-deps}
Assume that $0 < \eps \le \eps_0(\de)$.
Let $x,y\in X$. 
If $\eps d(x,y)\ge1$ then 
\begin{equation*}    
\exp(\eps d(x,y)) \simeq \frac{d_\eps(x,y)^2}{d_\eps(x)\,d_\eps(y)},
\end{equation*}
where the comparison constants depend only on $\de$, $M$ and $\eps_0$.
\end{cor}   

\begin{proof}
Since $\eps d(x,y)\ge1$, \eqref{eq-BHK-dist} can be written as
\begin{equation}   \label{eq-dist-(x,y)2}
\exp(-2\eps(x|y)_{z_0}) \simeq (\eps d_\eps(x,y))^2,
\end{equation}
where the comparison constants depend only on $\de$.
Moreover, \eqref{eq-BHK-d-rho} gives
\[
\exp(-\eps d(x,z_0)) \simeq \eps d_\eps(x)
\quad \text{and} \quad
\exp(-\eps d(y,z_0)) \simeq \eps d_\eps(y)
\]
with comparison constants depending only on $M$ and $\eps_0$.
Dividing \eqref{eq-dist-(x,y)2} by the last two formulas, and using the
definition of $(x|y)_{z_0}$ concludes the proof.
\end{proof}

We now wish to show that subWhitney balls
in the uniformization
$X_\eps$
are contained in balls of a fixed radius with respect to
the Gromov hyperbolic metric $d$ of~$X$.

\begin{thm}  \label{thm-subWhitney-balls}
For all $0<\eps\le \eps_0(\de)$, 
$x\in X$ and $0<r\le \tfrac12 d_\eps(x)$, we have
\[
B\biggl(x,\frac{C_1r}{\rho_\eps(x)}\biggr)
\subset B_\eps(x,r)
\subset B\biggl(x,\frac{C_2r}{\rho_\eps(x)}\biggr),
\]
where $C_1=e^{-(1+\eps M)}$ and $C_2=2e(2e^{\eps M}-1)$.
If $d_\eps(x,y)< C_1d_\eps(x)/2C_2$, then
\[
\frac{\rho_\eps(x)}{C_2}d(x,y)<d_\eps(x,y)\le e^{1/e}\rho_\eps(x)d(x,y).
\]
\end{thm}

\begin{remark}
As in Lemma~\ref{lem:dist-to-eps-bdry}, the constants $C_1$ and $C_2$
obtained for $\eps_0$ will do for $\eps<\eps_0$ as well.
The proof also shows that the
condition $0<r\le \tfrac12 d_\eps(x)$
  can be replaced by $0<r\le c_0 d_\eps(x)$
  for any fixed $0<c_0<1$,
but then $C_1$ and $C_2$ also
depend on $c_0$ and get progressively worse as $c_0$ approaches 1.
\end{remark}

\begin{proof}
Assume that $y\in B(x,C_1r/\rho_\eps(x))$ and 
let $\ga$ be a $d$-geodesic from $x$ to $y$.
The assumption $r\le\tfrac12 d_\eps(x)$ and \eqref{eq-BHK-d-rho}
then imply that for all $z\in\ga$,
\begin{equation}   \label{eq-def-C0}
d(x,z) \le d(x,y) <      \frac{C_1 r}{\rho_\eps(x)} \le
    \frac{C_1 d_\eps(x)}{2\rho_\eps(x)} \le 
    \frac{C_1 (2e^{\eps M}-1)}{2\eps} < \frac{C_1 e^{\eps M}}{\eps} =
    \frac{1}{\eps e}. 
\end{equation}
The triangle inequality then yields
$d(z,z_0)\ge d(x,z_0)-d(x,z) \ge d(x,z_0)-1/\eps e$ and hence
\[
\rho_\eps(z) = e^{-\eps d(z,z_0)} \le e^{1/e} \rho_\eps(x).
\]
From this and \eqref{eq-def-C0} it readily follows that
\begin{equation}\label{eq:Ural1}
d_\eps(x,y) \le \int_\ga \rho_\eps\,ds \le e^{1/e} \rho_\eps(x) d(x,y)
< C_1 e^{1/e} r < r.
\end{equation}

To see the other inclusion, assume that $d_\eps(x,y)<r \le \tfrac12 d_\eps(x)$ 
and let $\ga_\eps$ be 
a geodesic curve in $\clXeps$
connecting $x$ to $y$.
Then for all $z\in \ga_\eps$, we have by the triangle inequality that
\[
d_\eps(z) \ge d_\eps(x) - d_\eps(x,z) \ge d_\eps(x) - d_\eps(x,y)
> \tfrac12 d_\eps(x), 
\]
in particular $\ga_\eps\subset X_\eps$.
It now follows from~\eqref{eq-BHK-d-rho} that
\[ 
\rho_\eps(z) \ge \frac{\eps d_\eps(z)}{2e^{\eps M}-1} 
> \frac{\eps d_\eps(x)}{2(2e^{\eps M}-1)}
\ge \frac{\rho_\eps(x)}{C_2}, 
\] 
where $C_2$ is as in the statement of the theorem. This implies that
\begin{equation}\label{eq:Ammi1}
r > d_\eps(x,y) = \int_{\ga_\eps} \rho_\eps\,ds > \frac{\rho_\eps(x)}{C_2}d(x,y),
\end{equation}
and hence $B_\eps(x,r) \subset B(x,C_2r/\rho_\eps(x))$.

Finally, if $d_\eps(x,y)<C_1d_\eps(x)/2C_2$,
then from the last inclusion we see that 
$y\in B(x,C_1s/\rho_\eps(x))$ with $s=\tfrac12d_\eps(x)$.
Therefore we can apply~\eqref{eq:Ural1} and~\eqref{eq:Ammi1} to
obtain the last claim of the lemma.
\end{proof}

In this paper, the  letter $C$ will denote various positive
constants whose values may change
even within a line.
We 
write $Y \simle Z$ if there is an implicit
constant $C>0$ such that $Y \le CZ$,
and analogously
$Y \simge Z$ if $Z \simle Y$.
We also use the notation
$Y \simeq Z$ to mean
$Y \simle Z \simle Y$.
We will 
point out how the comparison constants
depend on various other constants related to
the metric measure spaces under study.

\section{Doubling property}
\label{sect-doubling}

\emph{In the rest of this paper, we will continue to assume that
$X$ is a locally compact roughly starlike Gromov hyperbolic space.
For general definitions and some results,
we will assume that $Y$ is a metric space equipped
with a Borel  measure $\nu$.}

\medskip

Just as for $X$, we will denote the metric on $Y$ by $d$, and balls in $Y$ by $B(x,r)$,
but it should always be clear from the context in which space
these concepts are taken.

\begin{deff}  
A Borel measure $\nu$, defined on a metric space $Y$, 
is \emph{globally doubling} if
\[
0 <\nu(B(x,2r))\le C_d \nu(B(x,r))<\infty
\]
whenever $x\in Y$ and $r>0$.
If this holds only for balls of radii $\le R_0$,
then we say that $\nu$
is \emph{doubling for balls of radii at most} $R_0$,
and also that $\nu$ is \emph{uniformly locally doubling}.
\end{deff}

The following result shows that the last condition
is independent of $R_0$, provided that $Y$ is quasiconvex.
Without assuming quasiconvexity this is not true as shown
by Example~\ref{ex-doubl-nonquasiconvex} below.

\begin{prop} \label{prop-doubling-R0-quasiconvex}
Assume that $Y$ is $L$-quasiconvex
and that $\nu$ is doubling on $Y$ for balls 
of radii at most $R_0$, with a doubling constant $C_d$.
Let $R_1>0$.
Then $\nu$ is doubling on $Y$ for balls 
of radii at most $R_1$
with a 
doubling constant depending only on $R_1/R_0$, $L$ and $C_d$.
\end{prop}

\begin{example} \label{ex-doubl-nonquasiconvex}
Let $X=([0,\infty) \times \{0,1\}) \cup (\{0\} \times [0,1])$
  equipped with the Euclidean distance and the measure
  $d\mu=w\, d\Lone$, where $\Lone$ is the Lebesgue measure and
\[
w(x,y)=\begin{cases}
1, & \text{if } y < 1, \\
e^x, & \text{if } y=1.
\end{cases}
\]
Then $X$ is a connected nonquasiconvex space and
$\mu$ is doubling for balls of radii at most $R_0$ if and only if $R_0 \le \frac12$.
This shows that the quasiconvexity
  assumption in Proposition~\ref{prop-doubling-R0-quasiconvex}
  cannot be dropped.
\end{example}

Before proving Proposition~\ref{prop-doubling-R0-quasiconvex}
we deduce the following lemmas.
In particular, Lemma~\ref{lem-cover-B-nr} covers
Proposition~\ref{prop-doubling-R0-quasiconvex}
under the extra assumption that $Y$ is a length space,
but with better control of the doubling constant than what is possible
in general quasiconvex spaces.

\begin{lem}   \label{lem-mu-doubl-X-doubl}
Assume that $\nu$ is doubling on $Y$ for
balls of radii at most $R_0$, with a doubling constant $C_d$.
Then every ball $B$ of radius $r\le \tfrac74R_0$ can be covered by at
most $C_d^7$ balls with centers in $B$ and radius $\frac17 r$.
\end{lem}

\begin{proof}  
Find a maximal pairwise disjoint collection of balls $B_j$ with centers in $B$
and radii $\tfrac{1}{14}r$.
Note that  for each $j$, 
\[
B_j \subset \tfrac{15}{14} B \quad \text{and} \quad
\tfrac{15}{112} B \subset \tfrac{127}{112}\cdot 14 B_j \subset 16B_j. 
\]
The doubling property then implies that
\[
\nu(\tfrac{15}{14} B) \le C_d^3 \nu(\tfrac{15}{112} B) \le C_d^7 \nu(B_j).
\]
From this and the pairwise disjointness of all $B_j$ we thus obtain
\[
\nu(\tfrac{15}{14} B) \ge \sum_j \nu(B_j)
  \ge \frac{1}{C_d^7} \nu(\tfrac{15}{14} B)  \sum_j 1,
\]
i.e.\ there are at most $C_d^7$ such balls.
As the balls $2B_j$ cover $B$, we are done.
\end{proof}

\begin{lem}   \label{lem-cover-B-nr}
Assume that $Y$ is a length  space
and that $\nu$ is doubling on $Y$ for balls 
of radii at most $R_0$, with a doubling constant $C_d$.
Let $n$ be a positive integer.
Then the following are true\/\textup{:}
\begin{enumerate}
\item  \label{it-doubl-along-geod}
If $x,x'\in Y$, $0<r\le R_0$ and $d(x,x')< nr$, then
\[
\nu(B(x',r)) \le C_d^n \nu(B(x,r)).
\]
\item \label{it-cover-nB}
Every ball $B$ of radius $nr$, with $r\le \tfrac14 R_0$, can be covered
by at most $C_d^{7(n+4)/6}$ balls of radius $r$, $n=1,2,\ldots$\,.
\end{enumerate}
In particular, for any $R_1>0$,
$\nu$ is doubling on $Y$ for balls of radii at most $R_1$
with a doubling constant depending only on $R_1/R_0$ and $C_d$.
\end{lem}

\begin{proof}
\ref{it-doubl-along-geod}
Connect $x$ and $x'$ by a curve of length $l_\ga < nr$.
Along this curve, we can find balls $B_j$ of radius $r$,
$j=0,1,\ldots,n$, such that $B_0=B(x,r)$, $B_n=B(x',r)$ and
$B_j\subset 2B_{j-1}$.
An iteration of the doubling property gives the desired estimate.

\ref{it-cover-nB}
Suppose that $\phi(n)$ is the smallest number  such that each
ball $B(x,nr)$ is covered by $\phi(n)$ balls $B_j$
of radius $r$.
Since $Y$ is a length space, the balls $7B_j$ cover $B(x,(n+6)r)$.
Using Lemma~\ref{lem-mu-doubl-X-doubl}, 
each $7B_j$ can in turn be covered by at most $C_d^7$
balls of radius $r$, which implies that 
$\phi(n+6)\le C_d^{7} \phi(n)$.
  As,  $\phi(1)=1$ and $\phi$ is nondecreasing,
  the statement follows by induction.
\end{proof}

\begin{proof}[Proof of Proposition~\ref{prop-doubling-R0-quasiconvex}]
We will use the inner metric $\din$, defined in \eqref{eq-din},
and denote balls with respect to $\din$ by $\Bin$.
It follows from the inclusions
\begin{equation} \label{eq-Bin}
    \Bin(x,r) \subset B(x,r) \subset \Bin(x,Lr),
\end{equation}
together with a repeated use of the doubling property for metric balls,
that $\nu$ is doubling  for inner balls  of radii at most $R_0$.
As $(X,\din)$ is a length space, it thus follows from Lemma~\ref{lem-cover-B-nr}
that $\nu$ is doubling  for inner balls  of radii at most $LR_1$.
Hence, using the inclusions \eqref{eq-Bin} again,
$\nu$ is doubling  for metric balls  of radii at most $R_1$.
\end{proof}

\section{The measure \texorpdfstring{$\mub$}{mubeta} is
  globally doubling
on \texorpdfstring{$X_\eps$}{X-epsilon}}\label{sect-doubling-of-mubeta}

\emph{Standing assumptions for this
section will be
    given after Example~\ref{ex:I-to-R}.}

  \medskip
  
Given a uniformly locally doubling measure $\mu$ on the Gromov hyperbolic space $X$,
we wish to obtain a globally doubling measure on its uniformization $X_\eps$. We do so as follows.

\begin{deff}  \label{def-muh-beta}
Assume
that $X$ is a
  locally compact roughly starlike Gromov hyperbolic space
  equipped with a Borel measure 
$\mu$, and that $z_0 \in X$.

Fix $\beta>0$, and set $\mu_\beta$ to be the measure on $X=X_\eps$ given by
\[
d\mu_\be = \rho_\be\,d\mu,
\quad \text{where } \rho_\be(x)=e^{-\beta d(x,z_0)}.
\]
We also extend this measure to $\clXeps$ by letting
$\mu_\be(\clXeps \setm X_\eps)=0$.
\end{deff}

Our aim in this section is to show that $\mu_\be$ is
a globally doubling measure on $\clXeps$,
under suitable assumptions
(see Theorem~\ref{thm-comp-largest-Whitney}).

Bonk--Heinonen--Koskela~\cite[Theorem~1.1]{BHK-Unif}
showed that there is a kind of duality between
Gromov hyperbolic spaces and bounded uniform domains,
see the introduction for further details.
Here we also equip these spaces with measures.
The following examples illustrate what happens
in a simple case.

\begin{example}\label{ex:R-to-I}
The Euclidean real line $X=\R$ is Gromov hyperbolic, because it is a
metric tree.
Since $\de=0$, any $\eps>0$ is allowed in the uniformization process,
by Theorem~\ref{thm-eps0}.
Setting $z_0=0$, we now determine what $X_\eps$ is.
For $x,y\in\R$, the uniformized metric is given by
\[
d_\eps(x,y)=\biggl| \int_x^y e^{-\eps|t|}\,dt \biggr| = 
   \begin{cases}   \displaystyle\frac{1}{\eps} |e^{-\eps|x|} - e^{-\eps |y|}|,   
                     &\text{if } xy\ge0, \\[3mm]
                   \displaystyle\frac{1}{\eps} (2 - (e^{-\eps|x|} + e^{-\eps|y|})),   
                     &\text{if }  xy\le0. 
\end{cases}
\]
With $y=0$ we get $d_\eps(x,0)=(1-e^{-\eps|x|})/\eps$. Hence the map 
$\Phi:X_\eps\to(-1/\eps,1/\eps)$ given by 
\[
\Phi(x)=\frac{1}{\eps}(1-e^{-\eps|x|})\sgn x
\]
is an isometry, identifying $X_\eps$ with the open interval $(-1/\eps,1/\eps)$.

However, when $X$ is equipped with the Lebesgue measure $\Lone$, 
the measure $\mu_\beta$ is not the Lebesgue measure on $(-1/\eps,1/\eps)$.
To determine $\mu_\be$, note that it is absolutely continuous with respect to
the Lebesgue measure $\Lone$ on the interval $(-1/\eps,1/\eps)$. 
So we compute the Radon--Nikodym
derivative of $\mu_\be$ with respect to $\Lone$. 
By symmetry, it suffices to consider $x>0$. Then
\[
d\mu_\be(\Phi(x))=e^{-\beta x}J_{\Phi^{-1}(\Phi(x))}\, d\Lone(\Phi(x))
  =e^{(\eps-\be)x}d\Lone(\Phi(x)).
\]
Substituting $\Phi(x)=z$ in the above, we get
\begin{equation}    \label{eq-mu-be-z}
d\mu_\beta(z)=(1-\eps z)^{-1+\beta/\eps}\, d\Lone(z) 
= (\eps d_\eps(z))^{-1+\beta/\eps}\, d\Lone(z),
\end{equation}
where $d_\eps(z)=1/\eps -z$ is the distance from $\Phi(x)=z\ge 0$ to
the boundary $\{\pm 1/\eps\}$ of $\Phi(X_\eps)$.

Similarly, if $X=\R$ is equipped with a weighted measure 
\[ 
d\mu(x)=w(x)\,d\Lone(x),
\] 
then as in~\eqref{eq-mu-be-z},
\begin{equation}   \label{eq-mube-weighted-R}
d\mu_\beta(z)  
  = (\eps d_\eps(z)) ^{-1+\beta/\eps} w(\Phi^{-1}(z))\, d\Lone(z).
\end{equation}
\end{example}

The following example reverses the procedure in
  Example~\ref{ex:R-to-I}.

\begin{example}\label{ex:I-to-R}
  The interval $X=(-1,1)$ is a uniform domain 
  and so, by
  Theorem~3.6 in Bonk--Heinonen--Koskela~\cite{BHK-Unif},
    it becomes a Gromov hyperbolic space when equipped with the
    quasihyperbolic metric $k$.
    The \emph{quasihyperbolic metric} is for $0\le y<z<1$ given by
\[
k(y,z)=\int_y^z\frac{1}{1-t}\, dt=\log\biggl(\frac{1-y}{1-z}\biggr),
\]
cf.\ Section~\ref{sect-hyperbolization}.
With $z_0=0$, by symmetry, we have $k(z,z_0)=\log(1/(1-|z|))$ for $z \in X$.
Hence we consider the map $\Psi:(-1,1)\to\R$ given by
\[
\Psi(z)=(\sgn z) \log \frac{1}{1-|z|},
\]
and see that $\Psi$ is an isometry between the Gromov hyperbolic space 
$(X,k)$ and the Euclidean line $\R$. 
By  Example~\ref{ex:R-to-I} with 
$\eps=1$, the
uniformization of $\R$ gives back the Euclidean interval $(-1,1)$.

We wish to find a measure $\mu$ on $(X,k)=\R$ such that the
weighted measure $\mu_\be$ given by Definition~\ref{def-muh-beta}
becomes the Lebesgue measure on  $(-1,1)$.
In view of \eqref{eq-mube-weighted-R} with $\eps=1$ and $\Phi=\Psi^{-1}$,
$\mu$ is
given by
$d\mu(x)=w(x)\,d\Lone(x)$, where
\[
w(x) = d_1(\Phi(x))^{1-\be} 
= (1-|\Phi(x)|)^{1-\be} = e^{(\be-1)|x|}.
\]
\end{example}

\medskip

\emph{In the rest of this  section, we assume that $X$ is a
  locally compact roughly starlike Gromov $\de$-hyperbolic space
  equipped with a measure 
$\mu$ which is doubling on $X$ for balls of radii at most $R_0$,
with a doubling constant $C_d$. We also fix a point $z_0 \in X$, let $M$ 
be the constant in the roughly starlike condition with respect to
$z_0$, and assume that 
\begin{equation}\label{eq:def-beta-naught}
0 < \eps \le \eps_0(\de)
\quad \text{and} \quad
\be > \be_0 := \frac{17 \log C_d}{3R_0}.
\end{equation}
Finally, we let $X_\eps$ be the uniformization of $X$ with
  uniformization center $z_0$.
}

\medskip 

In specific cases one may want to consider
how to optimally choose $R_0$, and the
corresponding $C_d$,  in the formula for $\beta_0$.
The factor $\frac{17}{3}$ comes from various estimates leading up
to the proof
of Proposition~\ref{prop-mu-be-xi},
and is not likely to be optimal.
The following example shows however that it is not too far from
optimal and that it cannot  be replaced by any constant $<1$.

\begin{example}
Let $X$ be the infinite regular $K$-ary metric tree, equipped
with the Lebesgue measure $\mu$,
as considered in
Bj\"orn--Bj\"orn--Gill--Shanmugalingam~\cite[Section~3]{BBGS}.
As it is a tree, any $\eps>0$ is allowed for uniformization.

If  $C_d(R)$ is the optimal doubling constant
for radii $\le R$, then a straightforward calculation shows
that
\[
     \lim_{R \to \infty} \frac{C_d(R)}{K^R}=1,
\]
and thus we are allowed, in this paper,
to use
any
\[
    \be > \frac{17 \log K^R}{3R} = \frac{17}{3} \log K. 
\]
In this specific case,
it was shown in \cite[Corollary~3.9]{BBGS}
that $\mu_\be$ is
globally
doubling and supports a global $1$-Poincar\'e inequality on $X_\eps$ whenever
$\be > \log K$.
For $\be  \le \log K$, $\mu_\be(X_\eps)= \infty$ and 
$\mu_\be$ cannot possibly be globally doubling on the bounded space $X_\eps$.
\end{example}

The following lemma gives us an estimate of $\mu_\be(B)$ for subWhitney balls
$B$.

\begin{lem}   \label{lem-mu-be-x}
Let $x\in X$ and $0<r\le\tfrac12 d_\eps(x)$.
Then
\[
\mu_\be(B_\eps(x,r)) \simeq \rho_\be(x)
   \mu\biggl(B\biggl(x,\frac{r}{\rho_\eps(x)}\biggr)\biggr)
\]
with comparison constants depending only on
$M$, $\eps$, $C_d$, $R_0$ and $\be$.
\end{lem}

\begin{proof}
  By Lemma~\ref{lem:dist-to-eps-bdry},
  we have for all $y\in B_\eps(x,r)$,
\begin{equation}  \label{eq-comp-rho-be-x-y}
\rho_\be(y) = \rho_\eps(y)^{\be/\eps} \simeq (\eps d_\eps(y))^{\be/\eps}
\simeq (\eps d_\eps(x))^{\be/\eps} \simeq \rho_\be(x).
\end{equation}
Moreover, Theorem~\ref{thm-subWhitney-balls} implies that 
\begin{equation}   \label{Beps-comp-B}
  B\biggl(x,\frac{C_1 r}{\rho_\eps(x)}\biggr)
    \subset B_\eps(x,r) \subset B\biggl(x,\frac{C_2 r}{\rho_\eps(x)}\biggr).
\end{equation}
This yields
\[
\mu_\be(B_\eps(x,r)) \simeq \rho_\be(x) \mu(B_\eps(x,r))
\simle \rho_\be(x) \mu\biggl(B\biggl(x,\frac{C_2 r}{\rho_\eps(x)}\biggr)\biggr)
\]
and similarly, 
$\mu_\be(B_\eps(x,r)) \simge \rho_\be(x) \mu(B(x,C_1 r/\rho_\eps(x))$.
Finally, Lemma~\ref{lem-cover-B-nr} shows that the last two balls in
$X$ have measure comparable to $\mu(B(x,r/\rho_\eps(x))$, which
concludes the proof.
\end{proof}

\begin{remark}
Lemma~\ref{lem-mu-be-x} implies that if $\mu_\beta$ is globally
doubling on $X_\eps$ then $\mu$ is uniformly locally doubling on $X$,
  i.e.\ the converse of Theorem~\ref{thm-main}\,\ref{ii-a} holds.
Indeed, if $0<r\le \frac{1}{4e\eps}$ and $x\in X$, then
$2r\rho_\eps(x) \le \tfrac12 d_\eps(x)$, by \eqref{eq-BHK-d-rho}.
Lemma~\ref{lem-mu-be-x}, with $r$ replaced by
$2r\rho_\eps(x)$ and  $r\rho_\eps(x)$, respectively, then gives
\[
\mu(B(x,2r)) \simeq
\frac{\mu_\be(B_\eps(x,2r\rho_\eps(x)))}{\rho_\be(x)}
\simeq \frac{\mu_\be(B_\eps(x,r\rho_\eps(x)))}{\rho_\be(x)}
\simeq \mu(B(x,r)).
\]
Similar arguments, combined with the arguments in the proof of
Lemma~\ref{lem-PI-for-subWhitney}, show that if $\mu_\be$ also supports
a global \p-Poincar\'e inequality on $X_\eps$ or $\clX_\eps$
then $\mu$ supports a
uniformly local \p-Poincar\'e inequality on $X$,
  i.e.\ the converse of Theorem~\ref{thm-main}\,\ref{ii-b} holds.
\end{remark}

We shall now estimate $\mu_\be(B)$ for balls $B$ centered at $\bdy_\eps X$
in terms of the (essentially) largest Whitney ball contained in $B$.
The existence of such balls is given by
Lemma~\ref{lem-ex-a0} below.

\begin{prop}  \label{prop-mu-be-xi}
Let $\xi\in\bdy_\eps X$ and $0<r\le2\diam_\eps X_\eps$.
Assume that $a_0>0$ and $z \in X$ are such that
$B_\eps(z,a_0r)\subset B_\eps(\xi,r)$ and $d_\eps(z)\ge 2a_0 r$.
Then, 
\[
\mu_\be(B_\eps(\xi,r))
\simeq \rho_\be(z) \mu(B(z,R_0)) 
\simeq \rho_\be(z)
\mu\biggl(B\biggl(z,\frac{a_0 r}{\rho_\eps(z)}\biggr)\biggr) 
\quad \text{and} \quad
\rho_\be(z) \simeq (\eps r)^{\be/\eps},
\]
where the comparison constants depend only on
$\de$, $M$, $\eps$, $C_d$,  $R_0$,  $\be$ and $a_0$.
\end{prop}

\begin{proof}
For $n=1,2,\ldots$\,, define the boundary layers
\[
A_n = \{ x\in B_\eps(\xi,r): e^{-n}r \le d_\eps(x) \le e^{1-n}r\}.
\]
Corollary~\ref{cor-comp-d-deps} implies that for every $x\in A_n$,
either $\eps d(x,z)<1$ or
\begin{equation*}
\exp(\eps d(x,z)) \simeq \frac{d_\eps(x,z)^2}{d_\eps(x)d_\eps(z)}
\le \frac{(d_\eps(x,\xi)+d_\eps(\xi,z))^2}{2a_0 e^{-n} r^2}
\le \frac{2e^n}{a_0},
\end{equation*}
and hence $\eps d(x,z) < n+C$, where $C$ depends only on
$\de$, $M$, $\eps$ and $a_0$.

Using Lemma~\ref{lem-cover-B-nr}\,\ref{it-cover-nB}, 
we can thus cover each layer $A_n\subset B(z,(n+C)/\eps)$ by  
$N_n\simle C_d^{14n/3\eps R_0}$ balls $B_{n,j}$
with centers in $B(z,(n+C)/\eps)$ 
and radius $R_0$.
Since $X_\eps$ is geodesic,
Lemma~\ref{lem-cover-B-nr}\,\ref{it-doubl-along-geod} 
implies that
each of these balls satisfies 
\[
\mu(B_{n,j}) \simle C_d^{n/\eps R_0} \mu(B(z,R_0)),
\] 
Moreover,
as in~\eqref{eq-comp-rho-be-x-y}
we see that
$\rho_\be(z) = \rho_\eps(z)^{\be/\eps}
\simeq (\eps d_\eps(z))^{\be/\eps} \simeq (\eps r)^{\be/\eps}$ and
\[ 
\rho_\be(x) = \rho_\eps(x)^{\be/\eps} \simeq (\eps d_\eps(x))^{\be/\eps}
\simeq (e^{-n} \eps r)^{\be/\eps}
\quad \text{for all }x\in A_n.
\] 
It thus follows that
\[
\mu_\be(A_n \cap B_{n,j}) \simle (e^{-n} \eps r)^{\be/\eps} \mu(B_{n,j})
\simle C_d^{n/\eps R_0} e^{-n \be/\eps} \rho_\be(z)\mu(B(z,R_0))
\]
and hence for $\beta>\be_0=17(\log C_d)/3R_0$,
\begin{align*}
\mu_\be(B_\eps(\xi,r)) 
&\le \sum_{n=1}^\infty \sum_{j=1}^{N_n} \mu_\be(A_n \cap B_{n,j}) \\
& \simle \rho_\be(z)
\mu(B(z,R_0))
    \sum_{n=1}^\infty (C_d^{17/3R_0})^{n/\eps} e^{-n\be/\eps} 
    \simeq \rho_\be(z)
    \mu(B(z,R_0)).
\end{align*}
Since $a_0 r\le  \tfrac12 d_\eps(z)$, Lemma~\ref{lem-mu-be-x}
implies that
\[ 
\mu_\be(B_\eps(\xi,r)) \ge \mu_\be(B_\eps(z,a_0 r)) \simeq \rho_\be(z) 
\mu\biggl(B\biggl(z,\frac{a_0 r}{\rho_\eps(z)}\biggr)\biggr).
\]
By~\eqref{eq-BHK-d-rho} we see that 
\[
r > d_\eps(z) \ge \frac{\rho_\eps(z)}{e\eps},
\]
and hence, by the doubling property for $\mu$ on $X$,
\[
\mu\biggl(B\biggl(z,\frac{a_0 r}{\rho_\eps(z)}\biggr)\biggr)
\ge \mu\Bigl(B\Bigl(z,\frac{a_0}{e\eps}\Bigr)\Bigr)
\simeq \mu(B(z,R_0)).
\qedhere
\]
\end{proof}

The following lemma shows how to pick $z$ and $a_0$ in
Proposition~\ref{prop-mu-be-xi}.

\begin{lem}   \label{lem-ex-a0}
  Let $0<a_0< a:=\min\{\tfrac{1}{8},\frac{1}{6A}\}$,
    where $A=A(\de)$ is as in Theorem~\ref{thm-eps0}.
    Then
  for every $x\in \clXeps$
and every $0<r\le 2\diam_\eps X_\eps$ we can find a ball 
$B_\eps(z,a_0 r) \subset B_\eps(x,r)$ such that $d_\eps(z)\ge 2a_0r$.
\end{lem}

\begin{proof}
First, assume that $x\in X_\eps$.
By Theorem~\ref{thm-eps0},
there is an $A$-uniform curve $\ga$ from $x$ to
$z_0$, parametrized by arc length $ds_\eps$. 
If $l_\eps(\ga) \ge \tfrac23r$ then  for $z=\ga(\tfrac13r)$
  we have
\[
d_\eps(z) \ge \frac{r}{3A} \quad \text{and} \quad
B_\eps\Bigl(z,\frac{r}{6A}\Bigr)
  \subset B_\eps\Bigl(x,\frac{r}{3}+\frac{r}{6A}\Bigr) \subset B_\eps(x,r).
\]
Thus, any $a_0\le \frac{1}{6A}$ will do in this case.
If $l_\eps(\ga) < \tfrac23r $, then letting $z=z_0$ yields
\[
B_\eps(z,\tfrac13 r) \subset B_\eps(x,l_\eps(\ga)+\tfrac13 r) \subset B_\eps(x,r),
\]
and for $a_0 \le\tfrac18$,
\[
d_\eps(z) = d_\eps(z_0) \ge 4a_0 \diam_\eps X_\eps \ge 2a_0r.
\]
This proves the lemma for $x\in X_\eps$. 
For $x\in \bdy_\eps X$ and any $0<a_0<a$, choose $r'=a_0r/a$ 
and $x'\in X_\eps$  
sufficiently close to $x$ so that, with the corresponding~$z$, 
\[
B_\eps(z,a_0 r) =  
B_\eps(z,a r')\subset B_\eps(x',r') \subset B(x,r)
\quad \text{and} \quad d_\eps(z)\ge 2ar' = 2a_0r.
\qedhere
\]
\end{proof}

Lemma~\ref{lem-mu-be-x} and Proposition~\ref{prop-mu-be-xi} can be
summarized in the following result, which roughly says that in 
$(X_\eps,\mu_\be)$, the measure of every ball is comparable to the 
measure of the (essentially) largest Whitney ball contained in it.

\begin{thm}   \label{thm-comp-largest-Whitney}
The measure $\mu_\be$ is globally doubling on $\clXeps$.

Moreover, with $a_0$ and $z$ provided by Lemma~\ref{lem-ex-a0},
we have for every $x\in \clXeps$ and $0<r\le2\diam_\eps X_\eps$,
\begin{equation}   \label{eq-comp-B-eps-Whitney}
\mu_\be(B_\eps(x,r)) \simeq \mu_\be(B_\eps(z,a_0 r)),
\end{equation}
where the comparison constants depend only on
$\de$, $M$, $\eps$, $C_d$,  $R_0$,  $\be$ and $a_0$.
\end{thm}

It follows directly that $\mu_\be$ is globally
  doubling also on $X_\eps$.
  The optimal doubling constants are the same,
  by Proposition~3.3 in Bj\"orn--Bj\"orn~\cite{BBnoncomp}.

\begin{proof}
We start by proving the measure estimate \eqref{eq-comp-B-eps-Whitney}.
Since $a_0 r\le \tfrac12 d_\eps(z)$,
Lemma~\ref{lem-mu-be-x} 
applied to $B_\eps(z,a_0 r)$ implies that
\begin{equation}   \label{eq-z-a0-r}
\mu_\be(B_\eps(z,a_0 r)) \simeq \rho_\be(z)
\mu\biggl(B\biggl(z,\frac{a_0 r}{\rho_\eps(z)}\biggr)\biggr).
\end{equation}
If $0<r\le\tfrac12 d_\eps(x)$ then by~\eqref{eq-comp-rho-be-x-y},
\eqref{Beps-comp-B} and Lemma~\ref{lem:dist-to-eps-bdry},
\[
\rho_\be(z) \simeq \rho_\be(x) 
\quad \text{and} \quad
d(x,z) \le \frac{C_2 r}{\rho_\eps(x)} \simle \frac12.
\] 
Lemma~\ref{lem-cover-B-nr} then implies  that 
\[
  \mu\biggl(B\biggl(z,\frac{a_0 r}{\rho_\eps(z)}\biggr)\biggr)
  \simeq \mu\biggl(B\biggl(z,\frac{r}{\rho_\eps(z)}\biggr)\biggr)
  \simeq \mu\biggl(B\biggl(x,\frac{r}{\rho_\eps(x)}\biggr)\biggr),
\]
and another application of Lemma~\ref{lem-mu-be-x}, this time to $B_\eps(x,r)$, 
proves \eqref{eq-comp-B-eps-Whitney} in this case.

If $r\ge\tfrac12 d_\eps(x)$ then $B_\eps(x,r) \subset B_\eps(\xi,3r)$
for some $\xi\in\bdy_\eps X$.
Proposition~\ref{prop-mu-be-xi} 
(with $a_0$ replaced by $\frac13 a_0$) 
then implies
\[
\mu_\be(B_\eps(\xi,3r)) \simeq 
\rho_\be(z)
\mu\biggl(B\biggl(z,\frac{a_0 r}{\rho_\eps(z)}\biggr)\biggr),
\]
which together with \eqref{eq-z-a0-r} proves 
\eqref{eq-comp-B-eps-Whitney} also in this case.

To conclude the doubling property, use the Whitney ball
$B_\eps(z,a_0 r)$ for both $B_\eps(x,r)$ and
$B_\eps(x,2r)$, with constants $a_0$ and $a'_0=\frac12 a_0$, respectively.
Since $d_\eps(z)\ge 2a_0r = 2a'_0 \cdot 2r$, we have by
\eqref{eq-comp-B-eps-Whitney},
first used with $a_0'$ and then with $a_0$,
\[
\mu_\be(B_\eps(x,2r)) \simeq \mu_\be(B_\eps(z,2a'_0r)) 
= \mu_\be(B_\eps(z,a_0 r)) \simeq \mu_\be(B_\eps(x,r)).\qedhere
\]
\end{proof}

We conclude this section with an estimate of the lower and upper
dimensions for the measure $\mu_\be$ at $\bdy_\eps X$.

\begin{lem}   \label{lem-s-on-bdy}
Let 
\[
s_\limpm =\frac{\be}{\eps} \pm \frac{\log C_d}{\eps R_0}.
\]
Then for all $\xi\in\bdy_\eps X$ and all\/ $0<r\le r'\le 2\diam_\eps X_\eps$,
\[
\Bigl( \frac{r}{r'} \Bigr)^{s_\limplus} \simle
\frac{\mu_\be(B_\eps(\xi,r))}{\mu_\be(B_\eps(\xi,r'))} 
\simle \Bigl( \frac{r}{r'} \Bigr)^{s_\limminus},
\]
with comparison constants depending only on
$\de$, $M$, $\eps$, $C_d$,  $R_0$,  $\be$ and the constant $a_0$
  from Theorem~\ref{thm-eps0}.
\end{lem}

Note that $0< s_\limminus \le s_\limplus$, because $\be> \be_0$,
where $\be_0$ is as in~\eqref{eq:def-beta-naught}.

\begin{proof}
Proposition~\ref{prop-mu-be-xi} and Lemma~\ref{lem-ex-a0} imply that
  there are $z,z' \in X$ such that
\begin{equation}   \label{eq-mube-mu-r0}
\mu_\be(B_\eps(\xi,r)) \simeq (\eps r)^{\be/\eps} \mu(B(z,R_0))
\quad \text{and} \quad 
\mu_\be(B_\eps(\xi,r')) \simeq (\eps r')^{\be/\eps} \mu(B(z',R_0)),
\end{equation}
where 
\begin{alignat*}{2}
  B_\eps(z,a_0r)&\subset B_\eps(\xi,r),
& \quad   d_\eps(z)& \ge 2a_0 r, \\
B_\eps(z',a_0r')&\subset B_\eps(\xi,r'),
& \quad  d_\eps(z')& \ge 2a_0 r'.
\end{alignat*}
From Corollary~\ref{cor-comp-d-deps} we conclude that
if $\eps d(z,z') \ge 1$ then
\[
\exp(\eps d(z,z')) \simeq \frac{d_\eps(z,z')^2}{d_\eps(z)d_\eps(z')}
\le \frac{(2r')^2}{(2a_0 r)(2a_0 r')} = \frac{r'}{a_0^2 r},
\]
and hence $d(z,z') \le \frac{1}{\eps}(C+\log(r'/r))$
holds regardless of the value of $d(z,z')$.
Lemma~\ref{lem-cover-B-nr}\,\ref{it-doubl-along-geod} 
with $n = \lceil d(z,z')/R_0 \rceil$
  (the smallest integer  $\ge d(z,z')/R_0$)
then implies that
\begin{equation}      \label{eq-comp-z-z'}
\frac{\mu(B(z,R_0))}{\mu(B(z',R_0))} \ge C_d^{-n}  
\simge \Bigl( \frac{r}{r'} \Bigr)^{(\log C_d)/\eps R_0},
\end{equation}
which together with \eqref{eq-mube-mu-r0} proves 
the first inequality in the lemma.
The second inequality follows similarly by interchanging $z$ and $z'$
in \eqref{eq-comp-z-z'}.
\end{proof}

\section{Upper gradients and Poincar\'e inequalities}\label{sect-defn-of-PI-ug}

\emph{We assume in this section that $1 \le  p<\infty$ 
and that $Y=(Y,d,\nu)$ is a metric space equipped
 with a  complete  Borel  measure $\nu$ 
 such that $0<\nu(B)<\infty$ for all balls $B \subset Y$.}

\medskip

We follow Heinonen and Koskela~\cite{HeKo98} in introducing
upper gradients as follows (in~\cite{HeKo98} they are referred to 
as very weak gradients).

\begin{deff} \label{deff-ug}
A Borel function $g:Y \to [0,\infty]$ is an \emph{upper gradient} 
of an extended real-valued function $u$
on $Y$ if for all 
arc length parametrized curves  
$\gamma: [0,l_{\gamma}] \to Y$,
\begin{equation} \label{ug-cond}
        |u(\gamma(0)) - u(\gamma(l_{\gamma}))| \le \int_{\gamma} g\,ds,
\end{equation}
where we follow the convention that the left-hand side is considered to be $\infty$ 
whenever at least one of the terms therein is $\pm \infty$.
If $g$ is a nonnegative measurable function on $Y$
and if (\ref{ug-cond}) holds for \p-almost every curve (see below), 
then $g$ is a \emph{\p-weak upper gradient} of~$u$. 
\end{deff}

We say that a property holds for \emph{\p-almost every curve}
if it fails only for a curve family $\Ga$ with \emph{zero \p-modulus}, 
i.e.\ there is a Borel function $0\le\rho\in L^p(Y)$ such that 
$\int_\ga \rho\,ds=\infty$ for every curve $\ga\in\Ga$.
The \p-weak upper gradients were introduced in
Koskela--MacManus~\cite{KoMc}. It was also shown therein
that if $g \in \Lploc(Y)$ is a \p-weak upper gradient of $u$,
then one can find a sequence $\{g_j\}_{j=1}^\infty$
of upper gradients of $u$ such that $\|g_j-g\|_{L^p(Y)} \to 0$.

If $u$ has an upper gradient in $\Lploc(Y)$, then
it has a \emph{minimal \p-weak upper gradient} $g_u \in \Lploc(Y)$
in the sense that for every \p-weak upper gradient $g \in \Lploc(Y)$ of $u$ we have
$g_u \le g$ a.e., see Shan\-mu\-ga\-lin\-gam~\cite{Sh-harm}
(or \cite{BBbook} or \cite{HKST}).
The minimal \p-weak upper gradient is well defined
up to a set of measure zero.

\begin{deff} \label{def.PI.}
$Y$ (or $\nu$)  supports a global \emph{\p-Poincar\'e inequality} if
there exist constants
$\lambda \ge 1$ \textup{(}called \emph{dilation}\textup{)}
and $\CPI>0$
such that for all balls $B \subset Y$, 
all bounded measurable functions $u$ on $Y$ and all upper gradients $g$ of $u$, 
\begin{equation} \label{PI-ineq}
        \vint_{B} |u-u_B| \,d\nu
        \le \CPI \diam(B) \biggl( \vint_{\lambda B} g^{p} \,d\nu \biggr)^{1/p},
\end{equation}
where $ u_B:=u_{B,\nu} 
 :=\vint_B u \,d\nu 
:= \int_B u\, d\nu/\nu(B)$.

If this holds only for balls $B$ of radii $\le R_0$,
then we say that $\nu$
supports a  \emph{\p-Poincar\'e inequality for balls of radii at most} $R_0$,
  and also that $Y$ (or $\nu$)
  supports a \emph{uniformly local \p-Poincar\'e inequality}.
\end{deff}

Multiplying bounded measurable functions by suitable cut-off functions 
and truncating integrable functions shows that one may
  replace ``bounded measurable'' by ``integrable'' in the definition.
  On the other hand, the proofs of \cite[Lemma~8.1.5 and Theorem~8.1.53]{HKST}
  show that
\eqref{PI-ineq} can equivalently be required  for all (not necessarily bounded)
  measurable 
functions $u$ on $\la B$ and all upper (or \p-weak upper) gradients
$g$ of $u$.
See also
\cite[Proposition~4.13]{BBbook}, \cite[Theorem~8.1.49]{HKST},
Haj\l asz--Koskela~\cite[Theorem~3.2]{HaKo},
Heinonen--Koskela~\cite[Lemma~5.15]{HeKo98}
and Keith~\cite[Theorem~2]{Keith}
for further equivalent versions.

\begin{thm} \label{thm-PI-arb-R0}
Assume that $\nu$ is doubling and
supports a \p-Poincar\'e inequality,
both properties holding for balls of radii at most $R_0$.
Also assume that $Y$ is $L$-quasiconvex and that $R_1 >0$.

Then $\nu$ 
supports a \p-Poincar\'e inequality, 
with dilation $L$,
for balls of radii at most $R_1$.
\end{thm}

The proof below can be easily adapted to show
that the same is true for so-called $(q,p)$-Poincar\'e inequalities.
The following examples show that the quasiconvexity assumption
cannot be dropped even if one assumes that $\nu$ is globally doubling,
and that one cannot replace $L$ in the conclusion by
the dilation constant in the assumed \p-Poincar\'e inequality, 
nor any fixed multiple of it.

\begin{example} \label{ex-PI-nonquasiconvex}
Let $X=([0,\infty) \times \{0,1\}) \cup (\{0\} \times [0,1])$
  equipped with the Euclidean distance and the 
  Lebesgue measure $\Lone$.
  Then $X$ is a connected nonquasiconvex space
and  $\Lone$ is globally doubling on $X$.
However, 
$\Lone$ supports a \p-Poincar\'e inequality on $X$, $p \ge 1$,
for balls of radii at most $R_0$ if and only if $R_0 \le 1$.
In this case one can choose the dilation constant $\la=1$.
This shows that the quasiconvexity assumption in
Theorem~\ref{thm-PI-arb-R0} cannot be dropped.
\end{example}

\begin{example}
For $a \ge1$, let
 $X=([0,a] \times \{0,1\}) \cup (\{0\} \times [0,1])$, 
  equipped with the Euclidean distance and the 
  Lebesgue measure $\Lone$.
  Then $X$ is a connected $(2a+1)$-quasiconvex space
and  $\Lone$ is globally doubling on $X$.
In this case, 
$\Lone$ supports a \p-Poincar\'e inequality on $X$, $p \ge 1$,
for balls of radii at most $R_0$ for any $R_0>0$,
with the optimal dilation
\[
\begin{cases}
1, & \text{if } R_0 \le 1, \\
\sqrt{1+a^2}, & \text{if } R_0 > 1.
\end{cases}
\]
This shows that the dilation constant
$L$ in the conclusion of Theorem~\ref{thm-PI-arb-R0}
cannot
in general be replaced 
  by the dilation constant in the \p-Poincar\'e inequality assumed
  for balls $\le R_0$,
 nor any fixed multiple of it.
\end{example}

\begin{proof}[Proof of Theorem~\ref{thm-PI-arb-R0}]
The arguments are similar to the proof of Theorem~4.4
in Bj\"orn--Bj\"orn~\cite{BBsemilocal}.
Let $C_d$, $\CPI$ and $\la$ be the constants
in the doubling property and the \p-Poincar\'e inequality
for balls of radii $\le R_0$.
Let $B$ be a ball with radius $r_B\le \tfrac52 LR_1=:R_2$.
We can assume that $r_B>R_0$.

First, note that the conclusions
in the first paragraph of the proof in \cite{BBsemilocal}
with $B_0=B$, $\sigma=L$, $r'=R_0/\la$ and $\mu$ replaced by $\nu$,
follow directly from our assumptions,
without appealing to Lemma~4.7 nor Proposition~4.8 in \cite{BBsemilocal}.
This and the use of Lemma~\ref{lem-mu-doubl-X-doubl}
explains why there is no need to assume properness here.

By Lemma~\ref{lem-cover-B-nr}, $\nu$ is doubling for balls
of radii $\le 7LR_2$, with doubling constant $C_d'$,
  depending only on $C_d$ and $LR_2/R_0$.
Hence, using Lemma~\ref{lem-mu-doubl-X-doubl}, we can cover $B$ by
at most $(C_d')^{7\lceil \log_7 (R_2/r')\rceil}$ 
balls $B'_j$ with radius $r'$.
Their centers can then be connected by $L$-quasiconvex curves.
As in the proof of \cite[Theorem~4.4]{BBsemilocal}, we then construct
along these curves a chain $\{B_j\}_{j=1}^N$ of balls of radius $r'$,
covering $B$ and with a uniform bound on $N$.
It follows that the constant $C''$ in the proof of
\cite[Theorem~4.4]{BBsemilocal}
only depends on $C_d$, $\CPI$, $\la$, $L$ and $R_2/R_0$.
Thus we conclude from the last but one displayed formula 
in the proof of  \cite[Theorem~4.4]{BBsemilocal} (with $B_0=B$)
that $\nu$ supports
a \p-Poincar\'e inequality for balls of radii at most $R_2$,
with dilation $2L$. 

That we can replace $2L$ by $L$ now follows
from \cite[Theorem~5.1]{BBsemilocal}, provided that
we decrease the bound on the radii to $R_1$.
\end{proof}

\section{Poincar\'e inequality on \texorpdfstring{$X_\eps$}{Xe}}\label{sect-PI-for-mubeta}

\emph{In this section, we assume that $X$ is a
  locally compact roughly starlike Gromov $\de$-hyperbolic space
  equipped with a Borel measure 
  $\mu$.
We also fix a point $z_0 \in X$, let $M$ 
be the constant in the roughly starlike condition with respect to
$z_0$,
and assume that 
\[ 
0 < \eps \le \eps_0(\de)
\quad \text{and} \quad
1 \le p < \infty.
\]
Finally, we let $X_\eps$ be the uniformization of $X$ with
  uniformization center $z_0$.
}

\medskip

The following lemma shows that the \p-Poincar\'e inequality
holds for $\mu_\be$ on sufficiently small subWhitney balls in $X_\eps$.
Recall that $\be_0= (17\log C_d)/3R_0$ as in~\eqref{eq:def-beta-naught}.

\begin{lem}  \label{lem-PI-for-subWhitney}
Assume that $\mu$ is doubling, with constant $C_d$, and
supports a \p-Poincar\'e inequality,  
with constants $\CPI$ and $\la$,
both properties holding for balls of radii at most $R_0$.
Let $\be > \be_0$.

Then there exists $c_0>0$, depending only on $\de$, $M$, 
$\eps$,  $R_0$ and $\la$,
such that for all $x\in X_\eps$ and all $0<r\le c_0 d_\eps(x)$, 
the \p-Poincar\'e inequality for $\mu_\be$
holds on $B_\eps=B_\eps(x,r)$,
i.e.\ for all bounded measurable functions $u$
and upper gradients $g_\eps$ of $u$ on $X_\eps$
we have
\[
\vint_{B_\eps} |u-u_{B_\eps,\mu_\be}| \,d\mu_\be 
   \le Cr \biggl( \vint_{\tau B_\eps} g_\eps^p \,d\mu_\be \biggr)^{1/p},
\]
where $\tau=C_2\la/C_1$, with $C_1$ and $C_2$ from 
Theorem~\ref{thm-subWhitney-balls}, 
and $C$ depends only on
$\de$, $M$, $\eps$, $C_d$,  $R_0$,  $\be$, $\la$ and $\CPI$.
\end{lem}

\begin{proof}
  Theorem~\ref{thm-subWhitney-balls} shows that if
  $c_0\le C_1/2C_2\la$ 
then
\begin{equation} \label{eq-balls}
B_\eps \subset B\biggl(x,\frac{C_2r}{\rho_\eps(x)}\biggr) =:B
\subset \la B\subset B_\eps\biggl(x,\frac{C_2\la r}{C_1}\biggr)
= \tau B_\eps.
\end{equation}
Moreover, as in~\eqref{eq-comp-rho-be-x-y} we have for all $y\in \tau B_\eps$,
\begin{equation}  \label{eq-comp-rho-be}
\rho_\be(y) = \rho_\eps(y)^{\be/\eps}
\simeq \rho_\eps(x)^{\be/\eps} =  \rho_\be(x).
\end{equation}
Hence, by Theorem~\ref{thm-comp-largest-Whitney}, all the balls
in \eqref{eq-balls}  have comparable $\mu_\be$-measures,
as well as  comparable $\mu$-measures.

Let $u$ be a bounded measurable function on $X_\eps$, or
  equivalently on $X$, and let
$g_\eps$ be an upper gradient of $u$ on $X_\eps$.
Since the arc length parametrization $ds_\eps$ with respect to $d_\eps$
satisfies $ds_\eps = \rho_\eps\,ds$, we conclude that for all
  compact rectifiable curves $\ga$ in~$X_\eps$,
\begin{equation} \label{eq-ds_eps}
    \int_\ga g_\eps \, ds_\eps = \int_\ga g_\eps \rho_\eps \, ds,
\end{equation}
and thus $g:= g_\eps\rho_\eps$ is an upper gradient of $u$ on $X$.
(Note that a compact curve in $X$ is rectifiable with respect to
  $d$ if and only if it is rectifiable with respect to $d_\eps$.)
If $c_0\le R_0\eps/C_2 (2e^{\eps_0M}-1)$, then 
by Lemma~\ref{lem:dist-to-eps-bdry},

\begin{equation*}   
\frac{C_2 r}{\rho_\eps(x)} \le \frac{C_2 c_0 d_\eps(x)}{\rho_\eps(x)}
\le \frac{C_2 c_0 (2e^{\eps_0M}-1)}{\eps} \le R_0,
\end{equation*}
and thus the \p-Poincar\'e inequality holds on $B$.
Using~\eqref{eq-comp-rho-be}  
we then obtain 
\begin{align*} 
\vint_{B_\eps} |u-u_{B,\mu}| \,d\mu_\be
&\simle \vint_{B} |u-u_{B,\mu}| \,d\mu 
   \le \frac{\CPI C_2 r}{\rho_\eps(x)}
           \biggl( \vint_{\la B} g^p \,d\mu \biggr)^{1/p} \\
&\simeq \frac{r}{\rho_\eps(x)}
           \biggl( \vint_{\la B} (g_\eps\rho_\eps)^p \,d\mu_\be \biggr)^{1/p}
\simle r \biggl(\vint_{\tau B_\eps} g_\eps^p \,d\mu_\be \biggr)^{1/p}.
\nonumber
\end{align*}
Finally, a standard argument \cite[Lemma~4.17]{BBbook}
  makes it possible to replace $u_{B,\mu}$ 
  on the left-hand side
by $u_{B_\eps,\mu_\be}$. 
\end{proof}

Bonk, Heinonen and Koskela~\cite[Section~6]{BHK-Unif} proved
that if $\Om$ is a locally compact uniform
space equipped with a measure $\mu$ such that $(\Om,\mu)$ is 
uniformly $Q$-Loewner in subWhitney balls,
then $\Om$ is globally $Q$-Loewner, where $Q>1$. 
If  $\mu$ is locally doubling with $\mu(B(x,r)) \simge r^Q$
whenever  $B(x,r)$ is a subWhitney ball,
then the local $Q$-Loewner condition is equivalent to
an analogous
local $Q$-Poincar\'e inequality, see~\cite[Theorems~5.7 and~5.9]{HeKo98}.

We have shown above that the measure $\mu_\beta$ on the uniformized 
space $X_\eps$ is globally doubling and supports 
a \p-Poincar\'e inequality for subWhitney balls.
Following the philosophy of~\cite[Theorem~6.4]{BHK-Unif},
the next theorem demonstrates that the \p-Poincar\'e inequality 
is actually global on $X_\eps$.

\begin{thm}    \label{thm-PI-Xeps}
Assume that $\mu$ is doubling and supports a \p-Poincar\'e
  inequality on $X$,
both properties holding for balls of radii at most $R_0$.
Let $\be > \be_0$ and $\la>1$.

Then $\mu_\be$ is globally doubling 
and supports a global \p-Poincar\'e inequality on $\clXeps$
with dilation~$1$, and
on $X_\eps$ with dilation~$\la$.
\end{thm}

If $X_\eps$ happens to be geodesic, then it follows from the proof
  below that we can choose the dilation constant $\la=1$ also on $X_\eps$.

\begin{proof}
The global doubling property follows from
Theorem~\ref{thm-comp-largest-Whitney}, both on $X_\eps$ and $\clXeps$.
Since $X_\eps$ is a length space and
Lemma~\ref{lem-PI-for-subWhitney} shows that the
\p-Poincar\'e inequality on $X_\eps$ holds for subWhitney balls,
the global \p-Poincar\'e inequality on $X_\eps$, with dilation $\la>1$,
  follows from the following proposition.
  Moreover, as $\clXeps$ is geodesic,
the global \p-Poincar\'e inequality on $\clXeps$, with dilation $1$,
  also follows from the following proposition.
\end{proof}

\begin{prop}    \label{prop-PI-to-global}
Let $(\Om,d)$ be a bounded $A$-uniform space equipped with a globally
  doubling measure $\nu$, which supports a \p-Poincar\'e inequality
  for all subWhitney balls corresponding to some fixed $0<c_0<1$.
Assume that $\Om$ is $L$-quasiconvex.
Then $\nu$ supports a global \p-Poincar\'e inequality on
$\Om$ with dilation~$L$.

If moreover the completion $\clOm$ of\/ $\Om$ is $L'$-quasiconvex,
then $\nu$, extended by $\nu(\bdy\Om)=0$,
supports a global \p-Poincar\'e inequality on
$\clOm$ with dilation~$L'$.
\end{prop}

Recall that $\Om$ is always $A$-quasiconvex by the
  $A$-uniformity condition, but that $L$ may be smaller than $A$.
Also, $\clOm$ is always $L$-quasiconvex, but 
it is possible to have $L'<L$.

\begin{proof}
Let $x_0\in \Om$, $0<r\le2\diam \Om$ and $B_0=B(x_0,r)$ be fixed.
The balls in this proof are with respect to $\Om$. 
It is well known, and easily shown using the arguments in the
proof of Lemma~\ref{lem-ex-a0}, that uniform spaces satisfy the
corkscrew condition, i.e.\ there exists $a_0$ (independent of $x_0$
and $r$) and $z$ such that
$d_\Om(z)\ge 2a_0r$ and $B(z,a_0r)\subset B_0$, cf.\ 
Bj\"orn--Shanmugalingam~\cite[Lemma~4.2]{BS-JMAA}.
With $c_0$ as in the assumptions of the proposition, let
\[
r_0= \frac{a_0 c_0 r}{8A} \le \frac{c_0 d_\Om(z)}{16A}
\quad \text{and} \quad r_i = 2^{-i}r_0, \quad i=1,2,\ldots.
\]
Since $\Om$ is $A$-uniform,
\cite[Lemma~4.3]{BS-JMAA} with $\rho_0=r_0$
and $\sig=1/c_0$ provides us for every $x\in B_0$ with a chain 
\[
\B_x= \{\Bij=B_\eps(\xij,r_i): i=0,1,\ldots \text{ and }
j=0,1,\ldots,m_i\}
\]
of balls connecting the ball $B_{0,0} :=  B(z,r_0)$ to $x$ as follows:
\begin{enumerate}
\item \label{first} 
For all $i$ and $j$ we have $m_i \le Ar/r_0 = 8A^2/a_0 c_0$,
\[
4r_i \le c_0 d_\Om(\xij)
\quad \text{and} \quad
d(\xij,x) \le 2^{-i}A d(x,z) < 2^{-i}Ar.
\]

\item \label{second}
For large $i$, we have $m_i=0$ and the balls $\Bio$ are centered at $x$.

\item \label{last}
The balls are ordered lexicographically, i.e.\ $\Bij$ comes before
$B_{i',j'}$ if and only if $i<i'$, or $i=i'$ and $j<j'$.
If $B^*$ denotes the immediate successor of $B\in \B_x$
then $B \cap B^*$ is nonempty.
\end{enumerate}
Let $u$ be a bounded measurable function on $\Om$
and $g$ be an upper gradient of $u$ in $\Om$. 
If $x\in B_0$ is a Lebesgue point of $u$ then
\begin{equation}
|u(x)-u_{B_{0,0}}| = \lim_{i\to\infty} |u_\Bio - u_{B_{0,0}}|
        \le \sum_{B\in\B_x} |u_{B^*} - u_B|,
                  \label{telescope}
\end{equation}
where $u_{B} = u_{B,\nu}$ and similarly for other balls.
Moreover, $B^*\subset3B$ and 
\[
|u_{B^*} - u_{B}| 
        \le |u_{B^*} - u_{3B}| + |u_{3B} - u_{B}|.
\]
As $3r_i \le c_0 d_\Om(\xij)$ and the radii of $B$ and $B^*$ 
differ by at most a factor 2, an application of the
  \p-Poincar\'e inequality on $3B$
shows that
\[
|u_{B^*} - u_{3B}| 
       \simle  \vint_{3B} |u-u_{3B}| \,d\nu
       \le C  r(B) \biggl( \vint_{3\la B} g^p \,d\nu \biggr)^{1/p},
\]
where $r(B)$ is the radius of $B$  and
$\la$ is the dilation constant in the assumed \p-Poincar\'e
  inequality for subWhitney balls.
The difference $|u_{3B} - u_{B}|$ is estimated in the same way.
Hence, inserting these estimates into (\ref{telescope}),
\[ 
|u(x)-u_{B_{0,0}}| \simle \sum_{B\in\B_x} r(B)
                       \biggl( \vint_{3\la B} g^p \,d\nu \biggr)^{1/p}.
\] 
We now wish to estimate the measure of level sets of the 
function $x\mapsto|u(x)-u_{B_{0,0}}|$ in $B_0$.
Assume that $|u(x)-u_{B_{0,0}}| \ge t$ and write 
$t= C_\al N t \sum_{i=0}^{\infty} 2^{-i\al}$,
where $\al\in(0,1)$ will be  chosen later,
and $N\le 1+Ar/R_0 = 1+8A^2/a_0 c_0 $ is the maximal number of balls
in $\B_x$ with the same radius.
Then
\[
C_\al N t \sum_{i=0}^{\infty} 2^{-i\al} = t
     \simle \sum_{B\in\B_x} r(B)
             \biggl( \vint_{3\la B} g^p \,d\nu \biggr)^{1/p}.
\]
Hence, there exists $B_x = B(\xij,r_i) \in \B_x$ such that 
\begin{equation*}
C_\al  2^{-i\al} t
     \simle r_i \biggl( \vint_{3\la B_x} g^p \,d\nu \biggr)^{1/p}.
\end{equation*}
We have $2^{-i} = r_i/r_0 = 8 Ar_i/a_0 c_0 r$,
and inserting this into the last inequality yields
\[
t \simle r \Bigl( \frac{r_i}{r} \Bigr)^{1-\al} 
          \biggl( \vint_{3\la B_x} g^p \,d\nu \biggr)^{1/p}.
\]
As $\nu$  is globally doubling, there exists 
$s>0$ independent of $B_x$ such that
\[
\frac{r_i}{r} \simle 
    \biggl( \frac{\nu(3\la B_x)}{\nu(B_0)} \biggr)^{1/s},
\]
see e.g.\ \cite[Lemma~3.3]{BBbook} or \cite[(3.4.9)]{HKST}.    
Hence 
\[
t \simle r \biggl( \frac{\nu(3\la B_x)}
                  {\nu(B_0)} \biggr)^{(1-\al)/s}
       \biggl( \vint_{3\la B_x} g^p \,d\nu \biggr)^{1/p},
\]
and choosing $\al\in(0,1)$ so that $\theta := 1- (1-\al)p/s \in(0,1)$, 
we obtain
\begin{equation}   \label{est-mu(Bx)}
\nu(3\la B_x)^{\theta} 
     \simle \frac{r^p }{t^p \nu(B_0)^{1-\theta}} 
            \int_{3\la B_x} g^p \,d\nu. 
\end{equation}
Let $E_t = \{x \in B_0: |u(x)-u_{B_{0,0}}| \ge t\}$ and $F_t$ be
the set of all points in $E_t$ which are Lebesgue points of $u$.
The global doubling property of $\nu$ guarantees that 
a.e.\ $x$ is a Lebesgue point of $u$, 
  see Heinonen~\cite[Theorem~1.8]{heinonen}.
  By the above,  for every $x\in F_t$ there exists $B_x \in\B_x$ satisfying
(\ref{est-mu(Bx)}). 
Note also that by construction of the chain,  
we have $x\in B'_x:=8(a_0 c_0)^{-1} A^2 B_x$.
The balls $\{B'_x\}_{x\in F_t}$, therefore cover $F_t$.
The 5-covering lemma (Theorem~1.2 in Heinonen~\cite{heinonen})
provides us with a pairwise disjoint collection
$\{\la\Bxip\}_{i=1}^{\infty}$ such that the union of all balls 
$5\la\Bxip$ covers $F_t$.
Then   the balls $3\la\Bxi\subset\la\Bxip$ 
are also pairwise disjoint and the 
global doubling property of $\nu$, together with 
\eqref{est-mu(Bx)}, yields
\begin{align*}
\nu(E_t) = \nu(F_t) &\le \sum_{i=1}^{\infty} \nu(5\la\Bxip)
      \simle \sum_{i=1}^{\infty} \nu(3\la\Bxi) \\
      &\simle \frac{r^{p/\theta}}
                   {t^{p/\theta} \nu(B_0)^{1/\theta-1}} 
           \sum_{i=1}^{\infty} \biggl( \int_{3\la \Bxi} g^p \,d\nu \biggr)
                    ^{1/\theta} \\
      &\le \frac{r^{p/\theta}}
                       {t^{p/\theta} \nu(B_0)^{1/\theta-1}} 
            \biggl( \int_{\La B_0} g^p \,d\nu \biggr)^{1/\theta},
\end{align*}
where $\La$ depends only on $A$, $\la$, $a_0$ and $c_0$.
Lemma~4.22 in Heinonen~\cite{heinonen}, which can be proved using the
Cavalieri principle, now implies that
\[
\vint_{B_0} |u-u_{B_{0,0}}| \,d\nu
       \simle r \biggl( \vint_{\La B_0} g^p \,d\nu \biggr)^{1/p}
\]
and a standard argument
\cite[Lemma~4.17]{BBbook}
allows us to replace $u_{B_{0,0}}$ by 
$u_{B_0}$. 

Since $\Om$ is $L$-quasiconvex,
it follows from \cite[Theorem~4.39]{BBbook}
that the dilation $\La$ in the obtained 
global \p-Poincar\'e inequality can be replaced by L.

Finally, by
  Proposition~7.1 in Aikawa--Shanmugalingam~\cite{AikSh05}
  (or the proof above applied within $\clOm$ and with $x_0 \in \clOm$),
  $\nu$ supports a global \p-Poincar\'e inequality,
  where, again using \cite[Theorem~4.39]{BBbook}, the dilation
  constant can be chosen to be $L'$.
\end{proof}

\section{Hyperbolization}
\label{sect-hyperbolization}

\emph{We assume in this section that $(\Om,d)$
    is a noncomplete $L$-quasiconvex space which
    is open in its completion $\clOm$,
    and let $\bdy \Om$ be its boundary within $\clOm$.}

\medskip

We define the \emph{quasihyperbolic metric} on $\Om$ by
\[
k(x,y) = \inf_\ga \int_\ga \frac{ds}{d_\Om(\ga(s))},
\quad \text{where } d_\Om(x)=\dist(x,\bdy \Om),
\]
$ds$ is the arc length parametrization of $\ga$, and
the infimum is taken
over all rectifiable curves in $\Om$ connecting $x$ to $y$.
It follows that $(\Om,k)$ is a length space.
Balls with respect to the quasihyperbolic metric $k$ will be denoted by $B_k$.

Even though our main interest is in hyperbolizing uniform spaces,
the quasihyperbolic metric makes sense in greater generality.
 In fact, the results in this section
hold also if we let $\Om \subsetneq Y$ be an $L$-quasiconvex open subset
of a (not necessarily complete) metric space $Y$
and the  quasihyperbolic metric $k$ is defined using
$d_\Om(x)=\dist(x,Y \setm \Om)$.

If $\Om$ is a locally compact uniform space,
then 
Theorem~3.6 in Bonk--Heinonen--Koskela~\cite{BHK-Unif}
 shows that the space $(\Om,k)$
is a proper geodesic Gromov hyperbolic space.
Moreover, if $\Om$ is bounded, then $(\Om,k)$ is roughly starlike.

As described in the introduction, 
the
operations of uniformization and hyperbolization are mutually
opposite, by Bonk--Heinonen--Koskela
\cite[the discussion before Proposition~4.5]{BHK-Unif}.

\begin{lem} \label{lem-k-est}
Let $x,y \in \Om$.
Then the following are true\/\textup{:}
\begin{alignat}{2}
  k(x,y) & \ge \frac{d(x,y)}{2d_\Om(x)}, &\quad& \text{if } d(x,y) \le d_\Om(x),
  \nonumber \\
  k(x,y) & \ge \tfrac{1}{2}, && \text{if } d(x,y) \ge d_\Om(x), \nonumber \\
 \frac{d(x,y)}{2d_\Om(x)} \le k(x,y) &\le \frac{2Ld(x,y)}{d_\Om(x)},
 && \text{if } d(x,y) \le \frac{d_\Om(x)}{2L}.
 \label{eq-k(x,y)}
\end{alignat}
Moreover,
\begin{alignat*}{2}
  B\biggl(x,\frac{r d_\Om(x)}{2L}\biggr)
  & \subset B_k(x,r)
  \subset B(x,2r d_\Om(x)),
  &\quad&   \text{if } r \le \frac{1}{2}, \\
  B_k\biggl(x,\frac{r}{2 d_\Om(x)}\biggr)
  & \subset B(x,r)
  \subset B_k\biggl(x,\frac{2Lr}{d_\Om(x)}\biggr),
  &\quad&   \text{if } r \le \frac{d_\Om(x)}{2L}.
\end{alignat*}
\end{lem}

If $\Om$ is $A$-uniform it is possible to get
an upper bound similar to the one in \eqref{eq-k(x,y)} also
when $d(x,y) \le \frac12 d_\Om(x)$, albeit with a little
more complicated expression for the constant.
As we will not need such an estimate, we leave it to the interested reader
to deduce such a bound.

\begin{proof}
Without loss of generality we assume that $x \ne y$.

Assume first that $d(x,y) \le d_\Om(x)$.
Let $\ga:[0,l_\ga] \to \Om$ be a curve from $x$ to $y$.
All curves in this proof will be arc length parametrized rectifiable curves in $\Om$.
Then $l_\ga \ge d(x,y)$ and
\[ 
\int_\ga \frac{ds}{d_\Om(\ga(s))}
\ge \int_0^{d(x,y)} \frac{dt}{d_\Om(x)+t} 
> \int_0^{d(x,y)} \frac{dt}{2d_\Om(x)} 
=  \frac{d(x,y)}{2d_\Om(x)}.
\] 
Taking infimum over all such $\ga$ shows that
$k(x,y) \ge d(x,y)/2d_\Om(x)$.

Suppose next that $d(x,y) \ge d_\Om(x)$.
Let $\ga:[0,l_\ga] \to \Om$ be a curve from $x$ to $y$.
Then $l_\ga \ge d(x,y) \ge d_\Om(x)$ and
\[ 
\int_\ga \frac{ds}{d_\Om(\ga(s))}
\ge \int_0^{d_\Om(x)} \frac{dt}{d_\Om(x)+t} 
> \int_0^{d_\Om(x)} \frac{dt}{2d_\Om(x)} 
=  \frac{1}{2}.
\] 
Taking infimum over all such $\ga$ shows that
$k(x,y) \ge \tfrac12$.

Assume finally that $d(x,y) \le d_\Om(x)/2L$.
As $\Om$ is $L$-quasiconvex,
there is a curve $\ga:[0,l_\ga] \to \Om$ from $x$ to $y$ with  length
$l_\ga \le  L d(x,y) \le \frac12 d_\Om(x)$.
Then
\[
k(x,y) \le \int_{\ga} \frac{ds}{d_\Om(\ga(s))}
   \le l_\ga \frac{2}{d_\Om(x)} 
   \le \frac{2Ld(x,y)}{d_\Om(x)}.
   \]

The ball inclusions now follow directly from this.
\end{proof}

We shall now equip  $(\Om,k)$ with a measure
determined by the original measure $\mu$ on $\Om$.
As before, for the results in this section it will be enough
to assume that $\Om$ is quasiconvex.

\begin{deff}  \label{def-hyp-meas-mu(B)}
Let $\Om$ be equipped with a Borel measure $\mu$.
For measurable $A\subset \Om$ and $\al>0$, let
\[
\mu^\al (A) = \int_A \frac{d\mu(x)}{d_\Om(x)^\al}.
\]
\end{deff}

\begin{prop}\label{prop:doubling-mu-alph}
  Assume that $\mu$ is globally
  doubling in $(\Om,d)$ with doubling constant $C_\mu$.
  Then $\mu^\alp$ is doubling for $B_k$-balls of radii at most $R_0=\frac18$,
  with doubling constant $C_d=4^\alp C_\mu^m$,
where $m=\lceil \log_2 8L \rceil$.

Moreover, if $R_1 >0$, then
$\mu^\alp$ is doubling for $B_k$-balls of radii at most $R_1$.
\end{prop}

\begin{proof}
  Let $x \in \Om$,
  $r \le \frac18$, $B_k=B_k(x,r)$ and $B=B(x,r d_\Om(x))$.
By Lemma~\ref{lem-k-est},
\[
  \mu^\alp(B_k)
  \ge \mu^\alp \biggl(\frac{1}{2L} B\biggr)
  \ge \biggl( \frac{1}{2d_\Om(x)}\biggr)^\alp \mu\biggl(\frac{1}{2L} B\biggr)
\]
and hence, again using Lemma~\ref{lem-k-est},
\begin{align*}
  \mu^\alp(2B_k)
   \le \mu^\alp(4B)
&  \le \biggl( \frac{2}{d_\Om(x)}\biggr)^\alp \mu(4B)
  \le \biggl( \frac{2}{d_\Om(x)}\biggr)^\alp C_\mu^m \mu\biggl(\frac{1}{2L} B\biggr) \\
  &\le \biggl( \frac{2}{d_\Om(x)}\biggr)^\alp (2d_\Om(x))^\alp C_\mu^m   \mu^\alp(B_k)
  = C_d   \mu^\alp(B_k).
\end{align*}

As $(\Om,k)$ is a length space, 
Lemma~\ref{lem-cover-B-nr} shows that
$\mu^\alp$ is doubling for $B_k$-balls of radii at most $R_1$
for any $R_1>0$.
\end{proof}

\begin{prop}\label{prop:PI-mu-alpha}
  Assume that $(\Om,d)$ is equipped with a globally doubling measure $\mu$ supporting a 
  global \p-Poincar\'e inequality with dilation $\la$ and $p \ge 1$.
  Let $\alp>0$ and $R_1>0$.
  Then $(\Om,k)$, equipped with the measure $\mu^\alp$, supports a
  \p-Poincar\'e inequality for balls of radii at most $R_1$ 
  with dilation $L$ and the other Poincar\'e constant depending only on $L$, $R_1$
 and the global doubling and Poincar\'e constants.
\end{prop}

\begin{proof}
Let $u$ be a bounded measurable function on $\Om$
and $\ghat$ be an upper gradient of $u$ with respect to $k$.
Since the arc length parametrization $ds_k$ with respect to $k$ satisfies
\[
ds_k = \frac{ds}{d_\Om(\,\cdot\,)},
\]
we conclude that
\[
    \int_\ga \ghat \, ds_k = \int_\ga \frac{\ghat}{d_\Om(\,\cdot\,)} \, ds
\]
and thus $g(z):= \ghat(z)/d_\Om(z)$
is an upper gradient of $u$ with respect to $d$,
see the proof of Lemma~\ref{lem-PI-for-subWhitney} for further details.

Next, let
$x\in\Om$, $0<r\le R_0:=1/8\la L$, $B_k=B_k(x,r)$ and
$B=B(x,2r d_\Om(x))$.
We see, by Lemma~\ref{lem-k-est}, that
\[
\frac{1}{4L} B \subset B_k \subset B
\quad \text{and} \quad
\la B \subset 4 \la L B_k,
\]
where all the above balls have comparable $\mu$-measures,
as well as comparable $\mu^\alp$-measures.
Note that $d_\Om(z) \simeq d_\Om(x)$ for all $z \in 4\la L B_k$.
Thus,
\begin{align*}
  \vint_{B_k}|u-u_{B,\mu}|\, d\mu^\alp
   & \simle \vint_B |u-u_{B,\mu}|\, d\mu 
    \simle r d_\Om(x) \biggl(\vint_{\la B} g^p \,d\mu\biggr)^{1/p} \\
   & \simeq r  \biggl(\vint_{\la B} \ghat^p \, d\mu\biggr)^{1/p}
    \simle r  \biggl(\vint_{4\la LB_k} \ghat^p \, d\mu^\alp\biggr)^{1/p}.
\end{align*}
A standard argument as in~\cite[Lemma~4.17]{BBbook}
  makes it possible to replace $u_{B,\mu}$
  on the left-hand side
by $u_{B_k,\mu^\alp}$,
and thus $Y$ supports a \p-Poincar\'e
  inequality for balls of radii $\le R_0$, with dilation $4\la L$.
The conclusion now follows from Theorem~\ref{thm-PI-arb-R0}.
\end{proof}

\begin{remark}
Let $X$ be a Gromov hyperbolic space, equipped with a measure $\mu$,
and consider its uniformization $X_\eps$, together with the measure
$\mu_\be$, $\be>0$, as in Definition~\ref{def-muh-beta}.
With $\al=\be/\eps$, it is then easily verified that the pull-back
to $X$ of the measure $(\mu_\be)^\al$,
defined on the hyperbolization $(X_\eps,k)$ of
$X_\eps$, is comparable to the original measure $\mu$.
\end{remark}

\section{An indirect product of Gromov hyperbolic spaces}
\label{unif-gromov-skewproduct}

\emph{We assume in this section that
$X$ and $Y$ are two locally compact roughly
    starlike Gromov $\delta$-hyperbolic spaces.
    We  fix two points $z_X \in X$ and $z_Y \in Y$,
    and let $M$ be a common constant
    for the roughly starlike conditions with respect to
$z_X$ and $z_Y$.
    We also assume that $0 < \eps \le \eps_0(\de)$
    and that $z_X$ and $z_Y$ serve as centers
    for the uniformizations $X_\eps$ and $Y_\eps$.
}

\medskip

In general, the Cartesian product $X\times Y$ of two Gromov
hyperbolic spaces $X$ and $Y$ need not be Gromov hyperbolic;
for example, $\R\times\R$ is is not Gromov hyperbolic.
In this section, we shall construct
an indirect product metric on $X\times Y$ that 
does give us a Gromov hyperbolic space, namely we
set $X\times_\eps Y$ to be the Gromov hyperbolic
space $(X_\eps\times Y_\eps,k)$.
To do so, we first need to show that the Cartesian product
of two uniform spaces, equipped with the sum of their metrics,
is a uniform domain.
This can be proved using Theorems~1 and~2 in
Gehring--Osgood~\cite{GeOs} together
with Proposition~2.14 in Bonk--Heinonen--Koskela~\cite{BHK-Unif},
but this would result in
a highly nonoptimal uniformity constant.
We instead give a more self-contained proof that also yields a better
estimate of the uniformity constant for the Cartesian product.

\begin{example}  \label{ex-R-times-R}
Recall that the uniformization $\R_\eps$ of the hyperbolic
$1$-dimensional space $\R$ is isometric to $(-\tfrac1\eps,\tfrac1\eps)$,
see Example~\ref{ex:R-to-I}.
Hence, for all $\eps>0$, $\R_\eps\times\R_\eps$ is a planar square
region, which is biLipschitz equivalent to the planar disk.
Thus also its hyperbolization $\R\times_\eps \R$ is biLipschitz
equivalent to the hyperbolic disk, which is the model 2-dimensional
hyperbolic space.
\end{example}

\begin{lem}  \label{lem-long-banana}
Let $(\Om,d)$ be a bounded $A$-uniform space.
Then for every pair of points $x,y\in\Om$ and for every $L$ with
$d(x,y)\le L \le \diam\Om$, there exists a curve $\ga\subset \Om$ of length
\begin{equation}   \label{eq-long-banana}
\frac{L}{5A} \le l(\ga) \le (A+1)L,
\end{equation}
connecting $x$ to $y$ and such that for all $z\in\ga$,
\[
d_\Om(z) \ge \frac{1}{16A^2} \min \{l(\ga_{x,z}),l(\ga_{z,y})\},
\]
where $\ga_{x,z}$ and $\ga_{z,y}$ are the subcurves
of $\ga$ from  $x$ to $z$
and from $z$ to $y$, respectively.
\end{lem}

\begin{proof}
Choose $x_0\in\Om$ such that
$d_\Om(x_0) \ge \tfrac45 \sup_{z\in\Om} d_\Om(z)$.
Then for all $z\in\Om$, with $\ga_{z,x_0}$ being an $A$-uniform curve from
$z$ to $x_0$, and $z'$  its midpoint,
\[
d(z,x_0) \le l(\ga_{z,x_0}) \le 2A d_\Om(z') \le \tfrac52 A d_\Om(x_0).
\]
Hence $\diam\Om \le 5A d_\Om(x_0)$.
Now, let $x,y\in\Om$ and $L$ be as in the statement of the lemma.
Let $\ga_{x,x_0}$ be an $A$-uniform curve from $x$ to $x_0$.
We shall distinguish two cases:

1.\ If $L\le 5Al(\ga_{x,x_0})$ then let
$\gah_x$ be the restriction of $\ga$ to $[0,L/10A]$
and $\xhat=\ga(L/10A)$ be its new endpoint.

2.\ If $L\ge 5Al(\ga_{x,x_0})$ then let
$\ga_x$ be the restriction of $\ga$ to $[0,\tfrac12 l(\ga_{x,x_0})]$
and $\xhat=\ga(\tfrac12 l(\ga_{x,x_0}))$ be its new endpoint. Note that
\[
d_\Om(\xhat) \ge d_\Om(x_0) - \frac{l(\ga_{x,x_0})}{2}
\ge \frac{\diam\Om}{5A} - \frac{L}{10A}
\ge \frac{L}{10A}.
\]
Choose a curve $\ga'$ of length $L/10A$, which starts and ends
at $\xhat$.
Then for all $z\in\ga'$,
\[
d_\Om(z) \ge d_\Om(\xhat) - \frac{L}{20A} \ge \frac{L}{20A}.
\]
Thus, concatenating $\ga'$ to $\ga_x$ we obtain a curve $\gah_x$ from $x$ to
$\xhat$ of length
\begin{equation}   \label{eq-length-gah-x}
\frac{L}{10A} \le l(\gah_x) \le \frac{L}{5A}
\end{equation}
and such that for all $z\in\gah_x$,
\begin{equation}   \label{eq-gah-x-uniform}
d_\Om(z) \ge  \frac{1}{\max\{4,A\}} l(\gah_{x,z}) \ge
\frac{1}{4A} l(\gah_{x,z}),
\end{equation}
where $\gah_{x,z}$ is the part of $\gah_x$ from $x$ to $z$.
The curve $\gah_x$, obtained in case 1, clearly satisfies
\eqref{eq-length-gah-x} and \eqref{eq-gah-x-uniform} as well.

A similar construction, using an $A$-uniform curve from $y$ to $x_0$,
provides us with a curve $\gah_y$ from $y$ to $\yhat$,
satisfying \eqref{eq-length-gah-x} and \eqref{eq-gah-x-uniform} with $x$
replaced by $y$.

Now, let $\gat$ be a uniform curve from $\xhat$ to $\yhat$ and let $\ga$
be the concatenation of $\gah_x$ with $\gat$ and $\gah_y$ (reversed).
Since $d(\xhat,\yhat) \le d(x,y) + 2L/5A \le (1+2/5A)L$,
we see that
\[
l(\ga) \le A \Bigl( 1+ \frac{2}{5A} \Bigr) L +
\frac{2L}{5A} \le (A+1)L
\]
and the right-hand side inequality
in~\eqref{eq-long-banana} holds, while the
left-hand side follows from~\eqref{eq-length-gah-x}.

To prove the second property, in view of \eqref{eq-gah-x-uniform},
it suffices to consider $z\in\gat$.
Without loss of generality, assume that
the part $\gat_{\xhat,z}$ of $\gat$ from $\xhat$
to $z$ has length at most $\tfrac12 l(\gat)$.
Note that \eqref{eq-gah-x-uniform},
applied to the choice $z=\xhat$, gives
\begin{equation}   \label{eq-z=xhat}
d_\Om(\xhat) \ge \frac{1}{4A} l(\gah_x).
\end{equation}
Again, we distinguish two cases.

1.\ If $\tfrac12 d_\Om(\xhat)\ge l(\gat_{\xhat,z})$ then
by \eqref{eq-z=xhat},
\[
d_\Om(z) \ge d_\Om(\xhat) - l(\gat_{\xhat,z}) \ge \tfrac12 d_\Om(\xhat)
\ge \max \Bigl\{ l(\gat_{\xhat,z}), \frac{1}{8A} l(\gah_x) \Bigr\},
\]
and hence we obtain that
\[
d_\Om(z) \ge \tfrac12 d_\Om(\xhat)
\ge \frac{1}{16A} (l(\gat_{\xhat,z}) + l(\gah_x))
= \frac{1}{16A} l(\ga_{x,z}),
\]
where $\ga_{x,z}$ is the part of $\ga$ from $x$ to $z$.

2.\ On the other hand, if $\tfrac12 d_\Om(\xhat)\le l(\gat_{\xhat,z})$
then by \eqref{eq-z=xhat} again,
\[
d_\Om(z) \ge \frac1A l(\gat_{\xhat,z}) \ge \frac{1}{2A} d_\Om(\xhat)
\ge \frac{1}{8A^2} l(\gah_x),
\]
We conclude that
\[
d_\Om(z) \ge \frac{1}{16A^2} (l(\gat_{\xhat,z}) + l(\gah_x))
= \frac{1}{16A^2} l(\ga_{x,z}).\qedhere
\]
\end{proof}

\begin{prop} \label{prop-unif-product}
Let $(\Om,d)$ and $(\Om',d')$ be two bounded uniform spaces,
with diameters and uniformity constants $D$, $D'$, $A$ and $A'$, respectively.
Then $\Omt=\Om \times\Om'$ is also a bounded uniform space with respect
to the metric
\begin{equation} \label{eq-dt}
\dt((x,x'),(y,y')) = d(x,y) + d'(x',y'),
\end{equation}
with  uniformity constant
\[
\tilde{A} = \frac{80 [(A+1)D+(A'+1)D']}{\min\{D/A^3,D'/(A')^3\}}.
\]
\end{prop}

\begin{proof}
The boundedness is clear.
Let $\xt=(x,x')$ and $\yt=(y,y')$ be two distinct points in $\Omt$,
and let
\[
\La = \max \biggl\{ \frac{d(x,y)}{D}, \frac{d'(x',y')}{D'} \biggr\}\le 1,
\quad L=\La D\ge d(x,y) \quad \text{and} \quad L'=\La D'\ge d'(x',y').
\]
Note that
\begin{equation} \label{eq-2}
\La(D+D') \ge \dt(\xt,\yt) \ge \La \min\{D,D'\} 
\ge  \La \min \biggl\{ \frac{D}{A^3},\frac{D'}{(A')^3} \biggr\}.
\end{equation}
We use Lemma~\ref{lem-long-banana} to find curves $\ga\subset\Om$ and
$\ga'\subset\Om'$, connecting $x$ to $y$ and $x'$ to $y'$,
respectively, of lengths
\begin{equation} \label{eq-3}
\frac{L}{5A} \le l(\ga) \le (A+1)L
\quad \text{and} \quad
\frac{L'}{5A'} \le l(\ga') \le (A'+1)L',
\end{equation}
and such that for all $z\in\ga$,
\begin{equation}    \label{eq-choose-ga-ga'}
d_\Om(z) \ge \frac{1}{16A^2} \min \{l(\ga_{x,z}),l(\ga_{z,y})\},
\end{equation}
where $\ga_{x,z}$ and $\ga_{z,y}$ are the parts of $\ga$ from  $x$ to $z$
and from $z$ to $y$, respectively; 
similar statements holding true for  $z'\in\ga'$ and $A'$.
Note that $\La>0$ since $\xt \ne \yt$.
Hence $L,L'>0$ and, by~\eqref{eq-3},
the curves $\ga$ and $\ga'$ are nonconstant.

Next, assuming that $\ga$ and $\ga'$ are arc length parametrized,
we show that the curve
\[
\gat(t) = \biggl( \ga \biggl(\frac{t}{l(\ga)} \biggr),
                 \ga' \biggl(\frac{t}{l(\ga')} \biggr)  \biggr),
\quad t\in[0,1],
\]
is an $\tilde{A}$-uniform curve in $\Omt$ connecting $\xt$ to $\yt$.
To see this, note that we have by the definition \eqref{eq-dt}
of $\dt$ that for all $0\le s\le t \le 1$, using \eqref{eq-3} and then \eqref{eq-2},
\begin{align*}
  l(\gat|_{[s,t]})
  =(t-s)(l(\ga)+l(\ga'))
&\le (t-s) [(A+1)D+(A'+1)D'] \La   \\
&\le (t-s) \tilde{A} 
\dt(\xt,\yt).
\end{align*}
In particular, $\gat$ has the correct length. Since
\[
\bdy\Omt = (\bdy\Om \times \Om') \cup (\Om \times \bdy\Om')
\cup (\bdy\Om \times \bdy\Om'),
\]
we see that for all $\gat(t)=(z,z')$ with $0 \le t \le \frac12$,
using \eqref{eq-choose-ga-ga'} and then \eqref{eq-3},
\[
d_\Om(z) \ge \frac{l(\ga)t}{16A^2}
\ge \frac{Lt}{80A^3}
= \frac{\La D t}{80A^3},
\]
and similarly $d_\Om(z') \ge \La D' t/80(A')^3$.
Thus, using \eqref{eq-3} for the last inequality,
\begin{align*}
  d_{\Omt}(\gat(t)) & = \min \{ d_\Om(z), d_{\Om'}(z') \} 
   \ge \frac{\La t}{80}\min\biggl\{\frac{D}{A^3},\frac{D'}{(A')^3}\biggr\} \\
  & = \frac{t}{\tilde{A}} [(A+1)L+(A'+1)L'] 
   \ge \frac{t}{\tilde{A}} [l(\ga)+l(\ga')] 
   =\frac{l(\gat_{\xt,\gat(t)})}{\tilde{A}}.
\end{align*}
As a similar estimate holds for $\frac12 \le t \le 1$,
we see that $\gat$ is indeed an $\tilde{A}$-uniform curve.
\end{proof}

We next see that
the projection map $\pi:X\times_\eps Y\to X$ given by $\pi((x,y))=x$
is Lipschitz continuous.

\begin{prop}   \label{prop-Lip-proj}
The above-defined projection map
$\pi:X\times_\eps Y\to X$  is $(C/\eps)$-Lipschitz continuous, with
$C$ depending only on $\eps_0$ and $M$.
\end{prop}

\begin{proof}
Since $X\times_\eps Y$ is geodesic,
it suffices to show that $\pi$ is locally $(C/\eps)$-Lipschitz
with $C$ independent of the locality.
With $C_1=e^{-(1+\eps M)}$ and $C_2=2e(2e^{\eps M}-1)$ as in 
Theorem~\ref{thm-subWhitney-balls}, for
$(x,y)\in X\times Y$ let
\[
r=\frac{C_1\min\{d_\eps(x),d_\eps(y)\}}{2C_2}, \quad
x'\in B_\eps(x,r)
\quad \text{and} \quad y'\in B_\eps(y,r). 
\]
The last part of Theorem~\ref{thm-subWhitney-balls}
together with Lemma~\ref{lem:dist-to-eps-bdry} then gives
\[
d(\pi(x,y),\pi(x',y'))=d(x,x') \simeq \frac{d_\eps(x,x')}{\rho_\eps(x)}
\simeq \frac{d_\eps(x,x')}{\eps d_\eps(x)},
\]
with comparison constants depending only on $\eps_0$ and $M$.

Let $k_\eps$ denote the quasihyperbolic metric 
on $\Om:=X_\eps\times Y_\eps$.
Note that since both $X_\eps$ and $Y_\eps$ are length spaces, so is $\Om$.
As $C_2/C_1 >2e$, we see that
\[
d_\Om((x,y)) = \min\{d_\eps(x),d_\eps(y)\} > 2e(d_\eps(x,x')+d_\eps(y,y'))
\]
and thus \eqref{eq-k(x,y)} in Lemma~\ref{lem-k-est} with $L=e$ 
yields
\begin{equation} \label{eq-keps}
k_\eps((x,y),(x',y')) \simeq
\frac{d_\eps(x,x')+d_\eps(y,y')}{\min\{d_\eps(x),d_\eps(y)\}}.
\end{equation}
It follows that
\[
d(\pi(x,y),\pi(x',y'))
\simle \frac{1}{\eps} k_\eps((x,y),(x',y')).\qedhere
\]
\end{proof}

Next, we shall see how $X\times_\eps Y$ compares to $X\times_{\eps^\prime}Y$.

\begin{prop}   \label{prop-canonical-Lip}
Let $0 < \eps' < \eps \le \eps_0(\de)$. The canonical identity maps
\[
\Phi:X\times_\eps Y\to X\times_{\eps^\prime}Y
\quad \text{and} \quad
\Psi:X_{\eps^\prime}\times Y_{\eps^\prime}\to X_\eps\times Y_\eps
\]
are Lipschitz continuous.
More precisely, there is a constant $C'$, depending only on $\eps_0$ and $M$,   such that
$\Phi$ is $(C'\eps'/\eps)$-Lipschitz while $\Psi$ is $C'$-Lipschitz.

Moreover, neither $\Phi^{-1}$ nor $\Psi^{-1}$ is Lipschitz continuous.
\end{prop}

\begin{proof}
We first consider $\Phi$.
  Since $X\times_\eps Y$ is geodesic,
it suffices to show that $\Phi$ is locally $(C'\eps'/\eps)$-Lipschitz
with $C'$ independent of the locality.
As in the proof of Proposition~\ref{prop-Lip-proj}, for
$(x,y)\in X\times Y$ and $C_1, C_2$ from
Theorem~\ref{thm-subWhitney-balls}, let
\[
r=\frac{C_1\min\{d_\eps(x),d_{\eps'}(x),d_\eps(y),d_{\eps'}(y)\}}{2C_2}, 
\quad x'\in B_\eps(x,r)
\quad \text{and} \quad y'\in B_\eps(y,r).
\]
Theorem~\ref{thm-subWhitney-balls} then gives
\begin{equation}  \label{eq-d-eps-rho}
d_\eps(x,x') \simeq \rho_\eps(x) d(x,x') \quad \text{and} \quad
d_\eps(y,y') \simeq \rho_\eps(y) d(y,y').
\end{equation}
Let $\dt_\eps$, $\dt_{\eps'}$, $k_\eps$ and $k_{\eps'}$ denote 
the product metrics as in \eqref{eq-dt} and
the quasihyperbolic metrics on $X_\eps\times Y_\eps$
and $X_{\eps'}\times Y_{\eps'}$ respectively.
As in \eqref{eq-keps}, we conclude that
\[
k_\eps((x,y),(x',y')) \simeq
\frac{d_\eps(x,x')+d_\eps(y,y')}{\min\{d_\eps(x),d_\eps(y)\}}.
\]
Without loss of generality we assume that $\rho_\eps(x)\le \rho_\eps(y)$, and then
using 
Lemma~\ref{lem:dist-to-eps-bdry},
\[
d_\eps(x) \simeq \frac{\rho_\eps(x)}{\eps} \le \frac{\rho_\eps(y)}{\eps}
\simeq d_\eps(y),
\]
in which case we also have that
\[
d_{\eps'}(x) \simeq \frac{\rho_{\eps'}(x)}{\eps'} \le \frac{\rho_{\eps'}(y)}{\eps'}
\simeq d_{\eps'}(y).
\]
Therefore, using 
\eqref{eq-d-eps-rho},
\begin{equation} \label{eq-k-with-rho}
k_\eps((x,y),(x',y')) \simeq \frac{d_\eps(x,x')+d_\eps(y,y')}{d_\eps(x)}
\simeq  \eps \biggl( d(x,x')+ \frac{\rho_\eps(y)}{\rho_\eps(x)} d(y,y') \biggr),
\end{equation}
with a similar statement holding true for $\eps'$.
Since
\[ 
\frac{\rho_{\eps'}(y)}{\rho_{\eps'}(x)}
= \biggl( \frac{\rho_\eps(y)}{\rho_\eps(x)} \biggr)^{\eps'/\eps}
\le \frac{\rho_\eps(y)}{\rho_\eps(x)},
\] 
we conclude from \eqref{eq-k-with-rho} that
\[
k_{\eps'}((x,y),(x',y'))
\simle \frac{\eps'}{\eps} k_\eps((x,y),(x',y')),
\]
which proves the Lipschitz continuity of $\Phi$.

We now compare the product uniform domains $X_\eps\times Y_\eps$ and
$X_{\eps^\prime}\times Y_{\eps^\prime}$.
With $(x,y), (x',y') \in X\times Y$
as in the first part of the proof, we have 
by \eqref{eq-d-eps-rho} and the assumption $0<\eps'<\eps$
that
\begin{align*}
\dt_\eps((x,y),(x',y')) &=d_\eps(x,x')+d_\eps(y,y') \\
&\simeq \rho_\eps(x)\, d(x,x')+ \rho_\eps(y)\, d(y,y')\\
&\le \rho_{\eps'}(x)\, d(x,x')+ \rho_{\eps'}(y)\, d(y,y') \\
&\simeq \dt_{\eps'}((x,y),(x',y')),
\end{align*}
which proves the Lipschitz continuity of $\Psi$.
On the other hand, choosing $y=y'=z_Y$, with $\rho_\eps(z_Y)=1$, gives
\[
\frac{\dt_{\eps'}((x,z_Y),(x',z_Y))}{\dt_\eps((x,z_Y),(x',z_Y))}
\simeq \frac{\rho_{\eps'}(x)}{\rho_\eps(x)}
=\rho_{\eps}(x)^{-1+\eps'/\eps}.
\]
Since $\eps' <\eps$, letting $d(x,z_X)\to\infty$ and so $\rho_\eps(x)\to 0$
shows that $\Psi^{-1}$ is \emph{not} Lipschitz. 

To show that $\Phi^{-1}$ is \emph{not} Lipschitz, let $x_j\in X$ be
such that $\rho_\eps(x_j)\to0$ (and equivalently,
$\rho_{\eps'}(x_j)\to0$) as $j\to\infty$.
With 
$C(\de)$ as in \eqref{eq-BHK-dist} and $C_1,C_2$
as in Theorem~\ref{thm-subWhitney-balls}, for $j=1,2,\ldots$ 
we choose $y_j\in Y$ such that
\[
d(z_Y,y_j)= \frac{C_1 d_\eps(x_j)}{4 C_2 C(\de)}.
\]
This is possible since $Y$ is geodesic.
Then, for sufficiently large $j$, we have
\(
\eps d(z_Y,y_j) \le 1
\)
and hence by \eqref{eq-BHK-dist},
\[
d_\eps(z_Y,y_j) \le C(\de) d(z_Y,y_j) = \frac{C_1d_\eps(x_j)}{4C_2}.
\]
Since also $\rho_\eps(x_j) \le 1 = \rho_\eps(z_Y)$,
we thus conclude from \eqref{eq-k-with-rho},
with the choice $x=x'=x_j$, $y=z_Y$ and $y'=y_j$, that
\[
k_\eps((x_j,z_Y),(x_j,y_j)) \simeq \frac{\eps d(z_Y,y_j)}{\rho_\eps(x_j)},
\]
with a similar statement holding also for $\eps'$.
This shows that
\[
\frac{k_\eps((x_j,z_Y),(x_j,y_j))}{k_{\eps'}((x_j,z_Y),(x_j,y_j))}
\simeq \frac{\eps \rho_{\eps'}(x_j)}{\eps'\rho_{\eps}(x_j)}
= \frac{\eps}{\eps'} \rho_{\eps}(x_j)^{-1+\eps'/\eps}  \to \infty,
\quad \text{as }j\to\infty.
\]
i.e.\ $\Phi^{-1}$ is not Lipschitz.
\end{proof}

\begin{remark}
If $X=Y=\R$ then, according to Example~\ref{ex-R-times-R}, all the
indirect products $\R \times_\eps \R$ are mutually biLipschitz
equivalent.
However, Proposition~\ref{prop-canonical-Lip} shows that this
equivalence cannot be achieved by the canonical identity map $\Phi$.
\end{remark}

By Theorem~1.1 in Bonk--Heinonen--Koskela~\cite{BHK-Unif},
$\Phi$ is biLipschitz if and only if
$\Psi$ is a quasisimilarity.
Note that $X_\eps$ and $X_{\eps^\prime}$ are quasisymmetrically
equivalent by~\cite{BHK-Unif},
and so are $Y_\eps$ and $Y_{\eps^\prime}$.
On the other hand, products of quasisymmetric maps need not be
quasisymmetric, as exhibited by the Rickman's rug
$([0,1],d_{\Euc})\times([0,1], d_{\Euc}^\alpha)$ for $0<\alpha<1$,
see Bishop--Tyson~\cite[Remark~1, Section~5]{BT} 
and DiMarco~\cite[Section~1]{DiM}.
This seems to happen whenever one of the component spaces 
has dimension~$1$ and the
other has dimension larger than~$1$.

\begin{example} \label{ex:non-qs}
Let $X$ be the unit disk in $\R^2$, equipped with the Poincar\'e
metric $k$, making it a Gromov hyperbolic space.
Let $Y=(-1,1)$ be equipped with the quasihyperbolic metric
(and so it is isometric to $\R$,
see Examples~\ref{ex:R-to-I} and~\ref{ex:I-to-R}).
For both $X$ and $Y$ we can choose $\eps=1$, resulting in
$X_1$ being the Euclidean unit disk and $Y_1$ being
the Euclidean interval $(-1,1)$.
Thus $X_1\times Y_1$ is a solid 3-dimensional Euclidean
cylinder, with boundary made up of $\Sphere^1\times[-1,1]$
  together with two copies of the disk.

Choosing
$0<\eps<1$, we instead obtain $X_\eps$ and $Y_\eps$, with $Y_\eps$
isometric to the Euclidean interval
$(-1/\eps,1/\eps)$, see Example~\ref{ex:R-to-I}.
The boundary of $X_\eps\times Y_\eps$ is made up of two copies of
$X_\eps$ together with $Z\times[-1/\eps,1/\eps]$, where $Z$ is the
$\eps$-snowflaking of $\Sphere^1$, which results
in $Z$ being biLipschitz equivalent to a generalized
von Koch snowflake loop.

If $X\times_\eps Y$ were biLipschitz equivalent to
$X\times_1 Y$, then 
$Z\times[-1/\eps,1/\eps]$
would be quasisymmetrically equivalent to a $2$-dimensional region in
$\partial (X_1\times Y_1)$, which is impossible as pointed out
before this example.
\end{example}

\section{Newtonian spaces and \texorpdfstring{\p}{p}-harmonic functions}
\label{sect-N1p-p-harm}

\emph{We assume in this section that $1 \le  p<\infty$ 
and that $Y=(Y,d,\nu)$ is a metric space equipped with a
complete  Borel  measure $\nu$ 
 such that $0<\nu(B)<\infty$ for all balls $B \subset Y$.}

\medskip

For proofs of the facts stated in this section 
we refer the reader to Bj\"orn--Bj\"orn~\cite{BBbook} and
Heinonen--Koskela--Shanmugalingam--Tyson~\cite{HKST}.

Following Shanmugalingam~\cite{Sh-rev}, 
we define a version of Sobolev spaces on $Y$. 

\begin{deff} \label{deff-Np}
For a measurable function $u:Y\to [-\infty,\infty]$, let 
\[
        \|u\|_{\Np(Y)} = \biggl( \int_Y |u|^p \, d\nu 
                + \inf_g  \int_Y g^p \, d\nu \biggr)^{1/p},
\]
where the infimum is taken over all  upper gradients $g$ of $u$.
The \emph{Newtonian space} on $Y$ is 
\[
        \Np (Y) = \{u: \|u\|_{\Np(Y)} <\infty \}.
\]
\end{deff}

In this paper we assume that functions in $\Np(Y)$
are defined everywhere (with values in $[-\infty,\infty]$),
not just up to an equivalence class in the corresponding function space.
This is important in Definition~\ref{deff-ug},
to make sense of $g$ being an upper gradient of $u$.
The space $\Np(Y)/{\sim}$, where  $u \sim v$ if and only if $\|u-v\|_{\Np(Y)}=0$,
is a Banach space and a lattice.
For a measurable set $E\subset Y$, the Newtonian space $\Np(E)$ is defined by
considering $(E,d|_E,\nu|_E)$ as a metric space in its own right.
We say  that $f \in \Nploc(\Om)$, where $\Om$ is an open subset of $X$, if
for every $x \in \Om$ there exists $r_x>0$ such that 
$B(x,r_x)\subset\Om$ and $f \in \Np(B(x,r_x))$.
The space $\Lploc(\Om)$ is defined similarly.

\begin{deff}
The (Sobolev) \emph{capacity} of a set $E\subset Y$  is the number 
\[
\Cp(E):=\CpY(E):=\inf_u \|u\|_{\Np(Y)}^p,
\]
where the infimum is taken over all $u\in \Np (Y) $ such that $u=1$ on $E$.
\end{deff}

A property is said to hold \emph{quasieverywhere}
(q.e.)\ if the set of all points 
at which the property
fails has $\Cp$-capacity zero. 
The capacity is the correct gauge 
for distinguishing between two Newtonian functions. 
If $u \in \Np(Y)$, then $u \sim v$ if and only if $u=v$ q.e.
Moreover, if $u,v \in \Nploc(Y)$ and $u= v$ a.e., then $u=v$ q.e.

We will also need the variational capacity.

\begin{deff}
Let $\Om\subset Y$ be open.
Then 
\[
  \Np_0(\Om):=\{u|_\Om: u \in \Np(Y) \text{ and } u=0 \text{ on } Y \setm \Om\}.
\]
The \emph{variational capacity} of $E\subset \Om$ with respect to $\Om$ is
\[
\cp(E,\Om) := \cpY(E,\Om):= \inf_u\int_{\Om} g_u^p\, d\nu,
\]
where the infimum is taken over all $u \in \Np_0(\Om)$ such that $u=1$ on $E$.
\end{deff}

The following lemma provides us with a sufficient condition
for when a set has positive capacity, in terms
of Hausdorff measures.
It is similar to Proposition~4.3 in Lehrb\"ack~\cite{Lehr}, but 
the dimension condition for $s$ is weaker here and 
is only required for $x\in K$.
For the reader's convenience, we provide a complete proof.
We will use Lemma~\ref{lem-hausdim-cap}
to deduce Proposition~\ref{prop-Hausdorff-Xeps}.

\begin{lem}  \label{lem-hausdim-cap}
  Let $(Y,d,\nu)$ be a complete metric space equipped with a
  globally doubling measure
$\nu$ supporting a global \p-Poincar\'e inequality.
  Let $E\subset Y$ be a Borel set of
positive $\kappa$-dimensional Hausdorff measure
and assume that for some $C,s,r_0>0$, 
\begin{equation}  \label{eq-assume-s}
  \nu(B(x,r)) \ge C r^s
  \quad \text{for all $x\in E$ and all $0<r\le r_0$}, 
\end{equation}
Then $\CpY(E)>0$  whenever $p>s-\kappa$.
\end{lem}

Note that if \eqref{eq-assume-s} holds for some $r_0$, then it holds
  with $r_0=1$, although $C$ may change.

\begin{proof}
By the regularity of the Hausdorff measure,
there is a compact set $K \subset E$ with
positive $\kappa$-dimensional Hausdorff measure.
Assume that $\CpY(K)=0$.
Then also the variational capacity $\cpY(K,B)=0$ for every ball $B\supset K$.
By splitting $K$ into finitely many pieces if necessary, and shrinking
$B$, we can assume that $\nu(2B\setm B)>0$.

As $\cpY(K,B)=0$, it follows from
\cite[Theorem~6.19]{BBbook} that there are
\[
u_k\in\Lip_0(B):=\{\phi \in \Lip(Y): \phi =0 \text{ in } Y \setm B\}
\]
with upper gradients $g_k$ such that $u_k=1$ on $K$,
$0 \le u_k \le 1$ on $Y$ and
\[
\int_Y g_k^p\,d\nu \to 0 \quad \text{as }k\to\infty.
\]

We can assume that $r_0\le \dist(K,Y\setm B)$ and set
$r_j=2^{-j}r_0$, $j=0,1,\ldots$\,.
For a fixed $x\in K$, consider the balls $B_j=B(x,r_j)$.
A standard telescoping argument, using the
doubling property of
$\nu$ together with the \p-Poincar\'e inequality,
then shows that for a fixed $k$ and $u:=u_k$,
\begin{align}  
|u(x) - u_{B_0}|
&\le \sum_{j=0}^\infty |u_{B_{j+1}}-u_{B_j}|
\simle \sum_{j=0}^\infty \vint_{B_j} |u-u_{B_j}| \,d\nu \nonumber \\
&\simle \sum_{j=0}^\infty \frac{r_j}{\nu(B_j)^{1/p}} 
\biggl(\int_{\la B_j} g_k^p\,d\nu \biggr)^{1/p}.
\label{eq-telescoping}
\end{align}
Because $u$ vanishes outside $B$ and $\nu(2B\setm B)>0$, we see
that
\[
u_{2B}:=\vint_{2B} u\,d\nu \le \frac{\nu(B)}{\nu(2B)}=:1-2\theta<1.
\]
Moreover,
\[
|u_{2B} - u_{B_0}|\simle \frac{r_B}{\nu(2\la B)^{1/p}} 
\biggl(\int_{2\la B} g_k^p\,d\nu \biggr)^{1/p} \to 0, \quad \text{as } k\to\infty,
\]
where $r_B$ stands for the radius of $B$. Since $u(x)=1$,
we conclude that for sufficiently large $k$, independently of $x\in K$,
\[
|u(x) - u_{B_0}| \ge |u(x)-u_{2B}| - |u_{2B}-u_{B_0}|
\ge 2\theta -|u_{2B}-u_{B_0}| \ge \theta
\simeq \sum_{j=0}^\infty r_j^{\tau},
\]
where $\tau=1-(s-\kappa)/p>0$.
Inserting this into \eqref{eq-telescoping} and comparing the sums, we
see that for each $x\in K$ there exists a ball $B_x=B_{j(x)}$ centered
at $x$ and with radius $r_x= r_{j(x)}$ such that
\begin{equation}   \label{eq-choose-Bx}
\int_{\la B_x} g_k^p\,d\nu \simge \frac{\nu(B_x)}{r_x^{p(1-\tau)}}
\simge r_x^{\kappa},
\end{equation}
because of the assumption~\eqref{eq-assume-s}.
Using the 5-covering lemma, we can out of the balls $\la B_x$
choose a
countable pairwise disjoint subcollection $\la\widehat{B}_j$,
$j=1,2,\ldots$\,, with radii $\hat{r}_j$, so that
$K\subset \bigcup_{j=1}^\infty 5\la\widehat{B}_j$.
Hence using \eqref{eq-choose-Bx} we obtain
\[
\sum_{j=1}^\infty \hat{r}_j^{\kappa}
\simle \sum_{j=1}^\infty \int_{\la \widehat{B}_j} g_k^p\,d\nu
\simle \int_{B} g_k^p\,d\nu \to0,
\quad \text{as } k\to\infty,
\]
showing that the $\kappa$-dimensional Hausdorff content (and thus
also the corresponding measure) is zero.
This causes a contradiction, which concludes the proof.
\end{proof}

\begin{deff} \label{deff-pharm}
Let $\Om \subset Y$ be open.
Then $u \in \Nploc(\Om)$ is \emph{\p-harmonic} in $\Om$ if
it is continuous and
\begin{equation}   \label{eq-def-p-harm}
\int_{\phi \ne 0} g_u^p \, d\nu
\le \int_{\phi \ne 0} g_{u+\phi}^p \, d\nu
\quad \text{for all } \phi \in \Np_0(\Om).
\end{equation}
\end{deff}

This is one of several equivalent definitions in the literature, 
see Bj\"orn~\cite[Proposition~3.2 and Remark~3.3]{ABkellogg}
(or \cite[Proposition~7.9 and Remark~7.10]{BBbook}).
In particular,
multiplying $\phi$ by suitable cut-off functions shows that
the inequality in \eqref{eq-def-p-harm} can
equivalently be required for all $\phi \in \Np_0(\Om)$
with bounded support.

If $\nu$ is locally doubling and supports a local \p-Poincar\'e
  inequality then every $u\in\Np\loc(\Om)$ satisfying
  \eqref{eq-def-p-harm} can be modified on a set of zero capacity to
  become continuous, and thus \p-harmonic, 
see Kinnunen--Shanmugalingam~\cite[Theorem~5.2]{KiSh1}.
Moreover, it follows from  
\cite[Corollary~6.4]{KiSh1} that
\p-harmonic functions obey the strong maximum
principle, i.e.\ if $\Om$ is connected, then
they cannot attain their maximum in $\Om$ without
being constant.

\begin{deff}\label{def:ann-qcvx}
A metric space $Y$ is 
\emph{locally annularly quasiconvex} around  a point $x_0$
if there exist $\La\ge2$ and $r_0>0$ such that for every $0<r\le r_0$, 
each pair of points $x,y  \in B(x_0,2r)\setm B(x_0,r)$
can be connected within  
$B(x_0,\La r)\setm B(x_0,r/\La)$ 
 by a curve of length at most $\La d(x,y)$.
\end{deff}

\begin{lem}  \label{lem-est-osc-energy}
  Let $Y$
  be a complete metric space equipped with a
  globally doubling measure
$\nu$ supporting a global \p-Poincar\'e inequality.
  Assume that a connected
  open set $\Om\subset Y$ is locally annularly quasiconvex around 
  $x_0\in\Om$, with parameters $\La$ and
  $r_0 < \dist(x_0,Y \setm \Om)/2\La$. 
Let $u$ be a \p-harmonic function in $\Om\setm\{x_0\}$. 
Then for every $0<r\le r_0$,
\begin{equation}   \label{eq-est-osc-energy}
\osc_{B(x_0,2r)\setm \{x_0\}} u 
 \le C \biggl( \sum_{k=0}^\infty \biggl(\frac{(2^{-k}r)^p}{\nu(B(x_0,2^{-k}r))}
   \biggr)^{1/(p-1)} \biggr)^{1-1/p} 
  \biggl( \int_{B(x_0,2\La r)} g_{u}^p\, d\nu \biggr)^{1/p},
\end{equation}
where $C$ depends only on $\La$ and
  the global doubling and Poincar\'e constants.
\end{lem}

Here $\dist(x_0,\emptyset)$ is considered to be $\infty$.

\begin{remark}
  Under the assumptions of Lemma~\ref{lem-est-osc-energy}
and the additional assumption $r < \frac14 \diam Y$, the sum
in \eqref{eq-est-osc-energy} is, by
e.g.\ \cite[Proposition~6.16]{BBbook},
comparable to
\[
\sum_{k=0}^\infty \cpY(B(x_0,2^{-k-1}r),B(x_0,2^{-k}r))^{1/(1-p)}
\le \cpY(\{x_0\},B(x_0,r))^{1/(1-p)},
\]
where the last inequality follows from Lemma~2.6 in
Heinonen--Kilpel\"ainen--Martio~\cite{HeKiMa} whose proof applies
verbatim also in the metric space setting.
Thus if $\cpY(\{x_0\},B(x_0,r))$ is positive, then
the above sum is finite.
\end{remark}

\begin{proof}[Proof of Lemma~\ref{lem-est-osc-energy}]
We can assume that $r_0 < \frac12 \diam \Om$.
  For $0<\rho\le r_0$, find $x,y \in B(x_0,2\rho)\setm B(x_0,\rho)$ so that
\begin{equation}   \label{eq-choose-x,y-osc}
|u(x)-u(y)| \ge \tfrac{1}{2} \osc_{B(x_0,2\rho)\setm B(x_0,\rho)} u.
\end{equation}
Let $\ga$ be a curve in the annulus $B(x_0,\La \rho)\setm B(x_0,\rho/\La)$
provided by the annular quasiconvexity.
Along this curve, we can find a chain of balls $\{B_j\}_{j=1}^{N}$
of radius $\rho/4\la\La$, such that
$N$ is bounded by a constant depending only on $\La$ and
the dilation $\la$ from the \p-Poincar\'e inequality, and
\begin{alignat*}{2}
  &2\la B_j\subset B(x_0,2\La \rho) \setm B(x_0,\rho /2\La),
  &\quad & \text{for }j=1,\ldots,N, \\
&B_j \cap B_{j+1} \ne \emptyset && \text{for }  j=1,\ldots,N-1.
\end{alignat*}
Using Lemma~4.1 in  Bj\"orn--Bj\"orn--Shanmugalingam~\cite{BBSliouville},
we thus get that
\[ 
   |u(x) -u(y)|  \le \sum_{j=1}^{N} \osc_{B_j} u  
                \simle  \frac{\rho}{4\la\La}
                    \sum_{j=1}^{N} \frac{1}{\nu(B_j)^{1/p}} 
                \biggl( \int_{2\la B_j} g_u^p\, d\nu \biggr)^{1/p}.
\]
Since $\nu$ is globally doubling, we have 
$\nu(B_j)\simeq \nu(B(x_0,\rho))$ and so by \eqref{eq-choose-x,y-osc}
and the uniform bound on $N$,
\[
\osc_{B(x_0,2\rho)\setm B(x_0,\rho)} u
\simle  \frac{\rho}{\nu(B(x_0,\rho))^{1/p}}
    \biggl( \int_{B(x_0,2\La \rho) \setm B(x_0,\rho/2\La)} g_u^p\, d\nu \biggr)^{1/p}.
\]
H\"older's inequality, together with the last
estimate applied to
$\rho=r_k=2^{-k}r$  
then yields
\begin{align*}
\osc_{B(x_0,2r)\setm \{x_0\}} u 
&\le \sum_{k=0}^\infty \osc_{B(x_0,2r_k)\setm B(x_0,r_k)} u \\
&\simle \biggl( \sum_{k=0}^\infty \biggl( 
 \frac{r_k}{\nu(B(x_0,r_k))^{1/p}} \biggr)^{p/(p-1)} \biggr)^{1-1/p} 
  \biggl( \sum_{k=0}^\infty \int_{A_k} 
           g_u^p\, d\nu \biggr)^{1/p},
\end{align*}
where $A_k=B(x_0,2\La r_k) \setm B(x_0,r_k/2\La)$.
These annuli have clearly bounded overlap depending only on $\La$, and so
\eqref{eq-est-osc-energy} follows.
\end{proof}

The following lemma will be used when proving
Theorem~\ref{thm-nonconst-energy-iff-pos-cap}.

\begin{lem}  \label{lem-annular}
Let $(\Om,d)$ be an $A$-uniform space and $a\in\bdy\Om$.
Then $\clOm\setm\{a\}$ is locally annularly quasiconvex around $a$
with $\La=4A$.
\end{lem}

\begin{proof}
Let $r>0$ and assume that $x,y\in \Om \cap B(a,2r) \setm B(a,\frac12 r)$.
Let $\ga$ be an arc length parametrized $A$-uniform curve joining $x$ to $y$.
Since $l_\ga\le A d(x,y)\le 4Ar$, we have $\ga\subset B(a,(2+2A)r)$.
Now, if $\tfrac14 r\le t\le l_\ga-\tfrac14 r$, then
\[
d(a,\ga(t))\ge d_\Om(\ga(t)) \ge \frac{1}{A} \min\{t,l_\ga-t\}
\ge \frac{r}{4A}.
\]
Similarly, if 
$d(x,\ga(t))<\tfrac14 r$  or $d(y,\ga(t))<\tfrac14 r$ then
$d(a,\ga(t)) > \tfrac14 r$.
In both cases it follows that $\ga \cap B(a,\tfrac14 r)=\emptyset$,
and the lemma is proved under the assumption that  $x,y\in\Om$.

Finally, if $x,y\in \clOm \cap B(a,2r) \setm B(a,r)$
with $x \ne y$, then we find 
\[
x',y'\in \Om \cap (B(a,2r) \setm B(a,\tfrac12 r))
\quad \text{with} \quad 
d(x,x')\le \frac{d(x,y)}{8A} 
\text{ and } 
d(y,y')\le \frac{d(x,y)}{8A}.
\]
By the definition of uniform space, $\Om$ is $A$-quasiconvex, and hence
so is $\clOm$.
Join $x'$ to $y'$ by a curve $\ga$ as in the first part of the proof.
Concatenating $\ga$ with the $A$-quasiconvex curves, joining $x$ to $x'$
and $y'$ to $y$, gives a suitable curve $\gat$ with length
\[
l_{\gat} \le A d(x',y') + 2\cdot \tfrac{1}{8} d(x,y)
\le (A+1)d(x,y),
\]
 which concludes the proof.
\end{proof}

\section{\texorpdfstring{\p}{p}-harmonic functions on
\texorpdfstring{$X$}{X} and \texorpdfstring{$X_\eps$}{Xe}}
\label{p-harm-appl}

\emph{In this  section, we assume that $X$ is a
locally compact roughly starlike Gromov $\de$-hyperbolic space
equipped with a complete  Borel  measure $\mu$ 
 such that $0<\mu(B)<\infty$ for all balls $B \subset X$.
We also fix a point $z_0 \in X$, let $M$ 
be the constant in the roughly starlike condition with respect to
$z_0$, and
assume that 
\[ 
0 < \eps \le \eps_0(\de),
\quad 
\be > 0 
\quad \text{and} \quad
1 \le p<\infty.
\]
Finally, we let $X_\eps$ be the uniformization of $X$ with
  uniformization center $z_0$.
When discussing the uniformization $X_\eps$,
and in particular $\CpXeps$, it will always be
assumed to be equipped with $\mu_\be$ for the $\be$ given above.
}

\medskip

In this section we shall see that with suitable choices of $p$, $\eps$ and
$\be$ satisfying $\be=p \eps$, each \p-harmonic function on the unbounded Gromov hyperbolic
space $(X,d,\mu)$ 
transforms into a \p-harmonic function on the bounded space
$(X_\eps,d_\eps,\mu_\be)$. This fact will make it possible 
to characterize, under uniformly local assumptions,
when there are no nonconstant \p-harmonic functions
with finite \p-energy on $X$, i.e.\ when the
\emph{finite-energy Liouville theorem} holds.
A function $u$ has \emph{finite \p-energy} with
respect to $(X,d,\mu)$  if $\int_X g_u^p \,d\mu<\infty$.

In the setting of complete metric spaces, equipped with a  globally
doubling measure supporting a global \p-Poincar\'e inequality, it was shown
in Bj\"orn--Bj\"orn--Shan\-mu\-ga\-lingam~\cite[Theorem~1.1]{BBSliouville} that
the finite-energy Liouville theorem holds on $X$
whenever
$X$ is either annularly quasiconvex around a point or
\[
\limsup_{r \to \infty} \frac{\mu(B(x_0,r))}{r^{p}} > 0
\quad \text{for some fixed point $x_0$.} 
\]

The focus of this section will be to consider
the finite-energy Liouville theorem for Gromov hyperbolic spaces
under uniformly local assumptions.

\begin{prop}   \label{prop-energy+Sob-spc}
Let $\Om \subset X$ be open and $u: \Om \to [-\infty,\infty]$ be measurable.
Then the following are true\/\textup{:}
\begin{enumerate}
\item \label{a-comp-energy}
With $g_u$ and $g_{u,\eps}$ denoting the minimal \p-weak upper gradients
of $u$ with respect to $(d,\mu)$ and $(d_\eps,\mu_\be)$,
respectively, we have 
\begin{equation}  \label{eq-gu-gesp}
g_{u,\eps}(x)=g_u(x) e^{\eps d(x,z_0)}
\end{equation}
and
\begin{equation}   \label{eq-comp-energy}
\int_\Om g_u(x)^p\,d\mu(x) 
= \int_\Om g_{u,\eps}(x)^p  e^{(\be-p\eps)d(x,z_0)}\,d\mu_\be(x).
\end{equation}
\item \label{a-Nploc}
$\Nploc(\Om,d,\mu)= \Nploc(\Om,d_\eps,\mu_\be)$.
\item \label{a-Np}
If\/ $\Om$ is bounded, then $\Np(\Om,d,\mu)= \Np(\Om,d_\eps,\mu_\be)$, 
as sets and with comparable norms\/
\textup{(}depending only on $\eps$, $\be$, $p$ and $\Om$\textup{)}. 
\end{enumerate}
\end{prop}

\begin{remark}
At first glance it would seem that the minimal \p-weak upper gradient $g_{u,\eps}$
of $u$ would also depend on the ambient measure $\mu_\be$, but because of the local
nature of minimal weak upper gradients and by the fact that the weight
$x\mapsto e^{-\beta d(x,z_0)}$ is locally bounded away from both $0$ and
$\infty$, it follows that $g_{u,\eps}$ indeed does not depend on 
the choice of $\be$, see the proof below.
\end{remark}

\begin{proof}[Proof of Proposition~\ref{prop-energy+Sob-spc}]
Clearly, \ref{a-Nploc} follows directly from \ref{a-Np}.
To prove \ref{a-comp-energy} and \ref{a-Np},
we conclude from \eqref{eq-ds_eps} that
$g_\eps(x):=g(x) e^{\eps d(x,z_0)}$ is an upper gradient of $u$ 
with respect to $d_\eps$
if and only if
  $g$ is an upper gradient of $u$ 
with respect to $d$.
Since \p-weak upper gradients can be approximated by upper gradients,
both in the $L^p$-norm and also
pointwise almost everywhere with respect to $\mu$ and (equivalently) $\mu_\beta$,
this identity holds also for \p-weak upper
gradients. 
In particular, \eqref{eq-gu-gesp} and \eqref{eq-comp-energy} hold, 
which proves part \ref{a-comp-energy}.

If $\Om$ is bounded, we also have that $\mu$ and $\mu_\be$ are
comparable on $\Om$, 
which implies that 
\[
\int_\Om |u|^p\,d\mu \simeq \int_\Om |u|^p\,d\mu_\be
\]
with comparison constants depending on $\be $ and $\Om$.
Together with \eqref{eq-comp-energy}, this implies that 
$u \in \Np(\Om,d,\mu)$ if and only if $u \in \Np(\Om,d_\eps,\mu_\be)$,
with comparable norms.
\end{proof}

\begin{remark}
The proof of Proposition~\ref{prop-energy+Sob-spc} also shows that
even if $\Om$ is not bounded  then for $\be\ge p\eps$,
\[
\|u\|_{\Np(\Om,d_\eps,\mu_\be)} \le \|u\|_{\Np(\Om,d,\mu)}
\]
and thus
\(
\Np(\Om,d,\mu) \subset \Np(\Om,d_\eps,\mu_\be).
\)
\end{remark}

\begin{prop}   \label{prop-pharm-d-deps}
Let $\Om \subset X$ be open.
If $p=\be/\eps>1$, 
then a function $u: \Om \to \R$ is \p-harmonic in $\Om$ with
respect to $(d,\mu)$ if and only if it is \p-harmonic in $\Om$ with
respect to $(d_\eps,\mu_\be)$.
Moreover, its \p-energy is the same in both cases, i.e.\
\begin{equation} \label{eq-pharm-energy}
\int_\Om g_u^p \,d\mu = \int_\Om g_{u,\eps}^p \,d\mu_\be.
\end{equation}
\end{prop}

\begin{proof}
By Proposition~\ref{prop-energy+Sob-spc}\,\ref{a-Nploc},
$u\in\Np\loc(\Om,d,\mu)$ if and only if $u\in\Np\loc(\Om,d_\eps,\mu_\be)$.
Let $\phi$ be a function with bounded support in $X$.
Then $\phi\in \Np_0(\Om,d,\mu)$ if and only if
$\phi \in \Np_0(\Om,d_\eps,\mu_\be)$.
Thus \eqref{eq-comp-energy}, together with a similar identity for
the minimal \p-weak upper gradients of $u+\phi$, shows that
\[
\int_{\phi \ne 0} g_u^p\,d\mu \le \int_{\phi \ne 0} g_{u+\phi}^p\,d\mu
\quad \text{if and only if} \quad
\int_{\phi \ne 0} g_{u,\eps}^p \,d\mu_\be
\le \int_{\phi \ne 0} g_{u+\phi,\eps}^p \,d\mu_\be.
\]
It then follows from the discussion after
Definition~\ref{deff-pharm}
that $u$ is \p-harmonic with respect to $(d,\mu)$
if and only if it is \p-harmonic with respect to $(d_\eps,\mu_\be)$.
Moreover, \eqref{eq-pharm-energy} follows directly from
\eqref{eq-comp-energy}.
\end{proof}

\begin{thm}   \label{thm-nonconst-energy-iff-pos-cap}
Assume that $\mu$ is doubling and
supports a \p-Poincar\'e inequality,
both properties holding for balls of radii at most $R_0$.
Let $\be > \be_0$ and 
assume that $p=\be/\eps>1$.
Then the following are equivalent\/\textup{:}
\begin{enumerate}
\item \label{it-Liouville-fails}
There exists a nonconstant \p-harmonic function on $(X,d,\mu)$
with finite \p-energy, i.e.\ the finite-energy Liouville theorem fails for
$X$.
\item  \label{it-2-pos-cap-sets}
  There are two disjoint compact sets $K_1,K_2 \subset \bdy_\eps X$
  with positive $\CpXeps$ capacity.
\end{enumerate}
\end{thm}

After proving the theorem we will give some
illustrating examples.
But first, before proving the theorem, we
will provide several useful characterizations of
the second condition~\ref{it-2-pos-cap-sets}.
The characterization \ref{j-pt},
applied  to the restriction of $\CpXeps$ to the boundary $\bdy_\eps X$, will be used 
in the proof of Theorem~\ref{thm-nonconst-energy-iff-pos-cap}.

\begin{lem} \label{lem-Capp-char}
Let $Z$ be a separable metric space,
and  $\Capp(\,\cdot\,)$ be  a monotone, countably subadditive set-function
with values in $[0,\infty)$, defined for all subsets of $Z$.
Assume that for each Borel set $E \subset Z$,
\begin{equation} \label{eq-Capp}
         \Capp(E)=\sup_K \Capp(K),
\end{equation}
where the supremum is taken over all compact subsets $K \subset E$.         

Define the \emph{support} of $\Capp$ as
\[
  \supp \Capp =
  \{x \in Z : \Capp(B(x,r))>0 \text{ for all } r >0\}.
\]

Then the following are equivalent\/\textup{:}
\begin{enumerate}
\item \label{j-cpt}
  There are two disjoint compact sets $K_1,K_2 \subset Z$
  such that $\Capp(K_1)>0$ and $\Capp(K_2)>0$.
\item \label{j-Borel}
  There is a Borel set $E \subset Z$ such that
  $\Capp(E)>0$ and $\Capp(Z \setm E)>0$.
\item \label{j-open}
  There is an open set $G \subset Z$ such that
  $\Capp(G)>0$ and $\Capp(Z \setm G)>0$.
\item \label{j-supp}
  The support $\supp \Capp$ contains at least two points.
\item \label{j-pt}
  $\Capp$ is not concentrated to one point, i.e.\
  $\Capp(Z \setm \{a\})>0$ for each $a \in Z$.
\end{enumerate}
\end{lem}

If $Y$ is equipped with a
globally doubling measure $\nu$ supporting a
global \p-Poincar\'e
inequality and $p>1$, then $\CpY$ is a Choquet capacity,
by \cite[Theorem~6.11]{BBbook},
and thus satisfies the assumptions above. 
Hence the assumptions are also satisfied for any 
restriction of $\CpY$ to any closed subset of $Y$ as well.
Example~6.6 in \cite{BBbook} shows that \eqref{eq-Capp}
can fail if $p=1$.

The assumption \eqref{eq-Capp} is only needed
to establish the equivalence of \ref{j-cpt} and \ref{j-Borel}.
On the other hand, separability is only used to deduce
the identity \eqref{eq-supp} below, which in turn
is used to show the equivalence of \ref{j-Borel}--\ref{j-pt}.

\begin{proof}
We start by showing that 
\begin{equation} \label{eq-supp}
    \Capp(Z \setm \supp \Capp)=0.
\end{equation}
To this end, for each $x \in Z \setm \supp \Capp$ there is $r_x>0$ so that
$\Capp(B(x,r_x))=0$.
As $Z$ is Lindel\"of (which for metric spaces is equivalent
to separability, see e.g.\ \cite[Proposition~1.5]{BBbook}),
we can write $Z \setm \supp \Capp$ as a countable union of
such balls, each of which has zero capacity.
Hence the countable subadditivity shows that \eqref{eq-supp} holds.

Now we are ready to prove the equivalences of \ref{j-cpt}--\ref{j-pt}.

  \ref{j-Borel} $\imp$ \ref{j-cpt}
By \eqref{eq-Capp} there are $K_1 \subset E$ and
$K_2 \subset Z \setm E$ such that $\Capp(K_1)>0$ and $\Capp(K_2)>0$.

\ref{j-cpt} $\imp$ \ref{j-open} $\imp$ \ref{j-Borel}
  These implications are trivial.

\ref{j-Borel} $\imp$ \ref{j-pt}
Let $a \in Z$.
If $a \in E$, then
$\Capp(Z \setm \{a\}) \ge \Capp(Z \setm E)>0$.
Similarly, if $a \notin E$, then
$\Capp(Z \setm \{a\}) \ge \Capp(E)>0$.

\ref{j-pt} $\imp$ \ref{j-supp}
As $\Capp(Z \setm \{a\})>0$ for each $a \in Z$,
it follows from \eqref{eq-supp}
that $\supp \Capp$ is nonempty.
Let $a \in \supp \Capp$.
As again $\Capp(Z \setm \{a\})>0$, and \eqref{eq-supp} holds, there
is $b \in \supp \Capp \setm \{a\}$.

\ref{j-supp} $\imp$ \ref{j-open}
Let $a,b \in \supp \Capp$, $a \ne b$,
and then let $G=B(a,\frac12 d(a,b))$.
Thus $\Capp(G)>0$ as $a \in \supp \Capp$,
while $\Capp(Z \setm G) \ge \Capp(B(b,\frac12 d(a,b)))>0$
since $b \in \supp \Capp$.
\end{proof}

\begin{proof}[Proof of Theorem~\ref{thm-nonconst-energy-iff-pos-cap}]
By Theorem~\ref{thm-PI-Xeps}, 
the uniformized space $(X_\eps,d_\eps,\mu_\be)$, as well as its closure $\clXeps$,
supports  a global \p-Poincar\'e inequality and $\mu_\be$ is
globally doubling.
Moreover, $\clXeps$ is complete.
It thus follows from \cite[Theorem~6.11]{BBbook},
that $\CpXeps$ is a Choquet capacity, and in particular
satisfies the assumptions in Lemma~\ref{lem-Capp-char},
and so does its restriction to $\bdy_\eps X$.

\ref{it-2-pos-cap-sets} $\imp$ \ref{it-Liouville-fails}
Let $f(x):= \dist_\eps(x,K_1)$. 
Since $\clXeps$ is bounded, we have $f\in\Np(\clXeps)$ and hence
there exists a \p-harmonic function $u$ in $X_\eps$ such that
$u-f\in\Np_0(X_\eps)$, see Shanmugalingam~\cite[Theorem~5.6]{Sh-harm}
(or \cite[Theorem~8.28 and Definition~8.31]{BBbook}). 
The function $u$ is denoted 
$H_pf$ in Bj\"orn--Bj\"orn--Shanmugalingam~\cite{BBS}
(and $Hf$ in \cite{BBbook}).
By the Kellogg property (\cite[Theorem~3.9]{BBS}
or \cite[Theorem~10.5]{BBbook}),
we have 
$\lim_{X_\eps \ni y\to x} u(y) = f(x)$ on $\bdy_\eps X$, except possibly
for a set of zero $\CpXeps$-capacity. 
Consequently, as $f= 0$ on $K_1$, $f>0$ on $K_2$ and both
$K_1$ and $K_2$ have positive $\CpXeps$-capacity,
$u$ must be nonconstant on $X_\eps$.

Proposition~\ref{prop-pharm-d-deps} implies that
$u\in\Np\loc(X,d,\mu) $ is \p-harmonic in $X$ with respect to
$(d,\mu)$ as well, and from~\eqref{eq-comp-energy} with $\be=p\eps$ 
it follows that
\[ 
\int_{X} g_u^p\,d\mu  = \int_{X_\eps} g_{u,\eps}^p  \,d\mu_\be 
\le \|u\|^p_{\Np(X_\eps,d_\eps,\mu_\be)} < \infty. 
\]

$\neg$ \ref{it-2-pos-cap-sets}    $\imp$ $\neg$ \ref{it-Liouville-fails}
By Lemma~\ref{lem-Capp-char}, there
is $a \in \bdy_\eps X$ such that $\CpXeps(\bdy_\eps X \setm \{a\})=0$.
The capacity of $\{a\}$ can be zero or positive.

Let $u$ be a \p-harmonic function in $(X,d,\mu)$  with finite \p-energy. 
Then $u$ is also \p-harmonic on $(X_\eps,d_\eps,\mu_\be)$
with finite \p-energy, by Proposition~\ref{prop-pharm-d-deps}. 
Applying the global \p-Poincar\'e inequality  to the ball
$B_\eps(x_0,2\diam_\eps X)\cap X_\eps=X_\eps$,
  with an arbitrary $x_0\in X$,
shows that $u\in \Np(X_\eps)$,
cf.\ \cite[Proposition~4.13\,(d)]{BBbook}.

If $\bdy_\eps X$ has zero $\CpXeps$-capacity
then it is removable for \p-harmonic functions in
$\Np(X_\eps)$,
by  Theorem~6.2 in Bj\"orn~\cite{ABremove} (or \cite[Theorem~12.2]{BBbook}).
Hence, an extension of $u$ is \p-harmonic on the
compact connected set $\clXeps$ and is thus constant by the
strong maximum principle. 

Finally, assume that $\CpXeps(\{a\})>0$.
Then $E:=\bdy_\eps X\setm\{a\}$ has zero capacity
and is  thus removable for \p-harmonic functions in $\Np(X_\eps)$, 
by  \cite[Theorem~6.2]{ABremove} (or \cite[Theorem~12.2]{BBbook}).
Since $u\in\Np(X_\eps)$, it follows that
an extension of $u$ is \p-harmonic 
in the open set $\clXeps\setm\{a\}=X_\eps\cup E$. 
By Lemma~\ref{lem-annular}, we know that
$X_\eps\cup E$ is annularly quasiconvex at $a$. 
Since $\CpXeps(\{a\})>0$,
it is also true that
$\cpXeps(\{a\},B_\eps(a,\rho))>0$ if $\rho < \frac{1}{4} \diam_\eps X_\eps$,
by e.g.\ \cite[Proposition~6.16]{BBbook}.
Moreover, $\clXeps$ is connected.
Thus Lemma~\ref{lem-est-osc-energy}, 
together with the remark after it, implies that for sufficiently small $r>0$,
\begin{equation*}   
\osc_{B_\eps(a,2r)\setm \{a\}} u 
 \simle \biggl( \int_{B_\eps(a,2\La r)} g_{u,\eps}^p\, d\mu_\be \biggr)^{1/p}.
\end{equation*}
Since $g_{u,\eps}\in L^p(\clXeps,\mu_\be)$, the last integral tends
to 0 as $r\to0$ and we conclude that 
$\lim_{\clXeps\setm\{a\}\ni y\to a}u(y)$ exists. 
In particular, $u$ is bounded on the compact set $\clXeps$. 
Finally, the strong maximum principle for \p-harmonic
functions on $\clXeps\setm\{a\}$ shows that $u$ must be constant on
$\clXeps$.
\end{proof}

\begin{example} (Continuation of Example~\ref{ex:R-to-I}.)
We have $C_d=2$, and all choices of $R_0$ are acceptable. Hence
any $\eps,\be>0$ are allowed.
Fixing $\eps>0$ and $1<p<\infty$
and choosing
$\be=p\eps$, we see that the weight in \eqref{eq-mu-be-z} becomes
\[
w(z)=\eps^{-1+\beta/\eps}(1/\eps-|z|)^{-1+\beta/\eps}=\eps^{p-1}(1/\eps-|z|)^{p-1}.
\]
By considering the functions
\[
u_j(z)=\begin{cases}
    \displaystyle
   \min\biggl\{1, \frac{1}{j}\log \frac{1}{1-\eps |z|} \biggr\}, &
     \displaystyle
  \text{if } |z| < \frac{1}{\eps}, \\[3mm]
   1, & 
    \displaystyle
   \text{if } |z| = \frac{1}{\eps},
   \end{cases}
\]
for which $\|u_j\|_{\Np(X_\eps,\mu_\be)} \to 0$, as $j \to \infty$,
we see that $\CpXeps(\{\pm 1/\eps\})=0$.

Note that $\R$ does not admit any nonconstant \p-harmonic function with finite
\p-energy.
\end{example}

\begin{example}\label{ex:weight-strip-to-diamond}
Consider $X=\R\times[-1,1]$, which is a Gromov hyperbolic space when equipped with the Euclidean
metric. 
We equip $X$ with a weighted measure 
\[
d\mu(x,y)=w(x,y)\, d\Ltwo(x,y)
\] 
such that
$(X,\mu)$ is uniformly locally doubling and supports a uniformly local \p-Poincar\'e inequality.
Fixing $z_0=(0,0)$,
the uniformization with $\eps=1$ gives a uniform domain
$X_1$ such that $\partial_1X$ consists of two points. 

To understand the potential theory and geometry
of $X_1$ near these two points, consider $z=(x,y)\in X$ such that $x\gg1$. 
Then, with $d_1$ denoting  the uniformized metric on $X_1$, we have
\[
d_1(z,z_0)\approx \int_0^x e^{-t}\, dt=1-e^{-x}
\quad \text{and} \quad
d_1((x,-1),(x,1))\approx e^{-x}.
\]
Here by $d_1(z,z_0)\approx 1-e^{-x}$ we mean that 
$d_1(z,z_0)/(1-e^{-x})\to 1$ as $x\to\infty$.
Thus, near the two boundary points, $X_1$ is (biLipschitz equivalent to)  
the diamond region in $\R^2$ with corners
$(\pm1,0)$ and $(0,\pm 1)$.
The two boundary points of $X_1$ are $\zpm:=(\pm 1,0)$.

Let  $\be =p >1$ and let
$\mu_\beta$ be the weighted measure
on $X_1$, given by  Definition~\ref{def-muh-beta}.
By Theorem~\ref{thm-nonconst-energy-iff-pos-cap},
$X$ supports a 
nonconstant \p-harmonic function with finite \p-energy if and only if
both boundary points $\zpm$ have positive
$\CpXone$-capacity.
By Bj\"orn--Bj\"orn--Lehrb\"ack~\cite[Proposition~5.3]{BBLintgreen},
$\CpXone(\{\zplus\})>0$ if and only if
\[
\int_0^{r_0} \biggl(\frac{r^p}{\mu_\be(B_1(\zplus,r))}\biggr)^{1/(p-1)}
\frac{dr}{r}<\infty
\]
for some (all) sufficiently small $r_0$,
where the balls $B_1(\zplus,r)$ are with respect to the metric $d_1$.
By the global doubling property of $\mu_\be$ we see that
\[
\mu_\be(B_1(\zplus,r))\simeq \mu_\be(B_1(\zplus,r)
\setminus B_1(\zplus,\tfrac12 r)).
\]
In view of~\eqref{eq-BHK-d-rho}, each of these annuli is (roughly) the
image of a rectangular region with fixed size and at distance
approximately $\log(1/r)$ from the base point $z_0$.
Letting $Q(t)=[t-1,t+1]\times[-1,1]$, we therefore have
\[
\mu_\be(B_1(\zplus,r)) \simeq \int_{Q(\log(1/r))} e^{-\beta x}\, w(x,y)\, d\Ltwo(x,y).
\]
Since $e^{-\beta x} \simeq r^\beta$ on $Q(\log(1/r))$,
we therefore conclude that
$\CpXone(\{\zplus\})>0$ if and only if
\[
\int_0^{r_0} \biggl( \int_{Q(\log(1/r))} w\,d\Ltwo \biggr)^{1/(1-p)} 
   \frac{dr}{r}
<\infty, 
\]
or equivalently,
\begin{equation} \label{eq-char-liouv-strip}
\int_{0}^\infty \biggl( \int_{Q(t)} w\,d\Ltwo \biggr)^{1/(1-p)} \,dt 
<\infty.
\end{equation}
An analogous condition holds for $\zminus$.

Note that when $w\equiv1$, both \eqref{eq-char-liouv-strip} and its
analogue for $\zminus$ fail, showing that the unweighted strip
$\R\times[-1,1]$ satisfies the finite-energy Liouville theorem.
This special case was obtained in
Bj\"orn--Bj\"orn--Shanmugalingam~\cite{BBSliouville}
by a more direct method, without the use of uniformization.
\end{example}

\begin{remark}
The weighted Euclidean real line $(\R,\mu)$, where $d\mu=w\,dx$
is uniformly locally doubling and supports a uniformly local
\p-Poincar\'e inequality, can be treated similarly and we obtain that
$(\R,\mu)$ supports nonconstant \p-harmonic functions with finite
energy if and only if a condition similar to
\eqref{eq-char-liouv-strip} holds on it:
\begin{equation}  \label{eq-cond-R}
\int_0^\infty 
\biggl( \int_{t-1}^{t+1} w(x)\,dx \biggr)
^{1/(1-p)} \,dt  <\infty.
\end{equation}
In~\cite{BBSliouville}, this question
was studied by different methods and under local assumptions on $w$.
It follows from the results in 
Bj\"orn--Bj\"orn--Shanmugalingam~\cite{BBSpadm} on local $A_p$
weights, that the condition
\[
\int_0^\infty w(x)^{1/(1-p)}\,dx <\infty,
\]
obtained in~\cite{BBSliouville}, is equivalent to \eqref{eq-cond-R}
under the local assumptions on $w$.
\end{remark}

We end the paper with the following result which is a
direct consequence of
Theorem~\ref{thm-PI-Xeps} together with
Lemmas~\ref{lem-s-on-bdy} and~\ref{lem-hausdim-cap}. 
In combination with  Theorem~\ref{thm-nonconst-energy-iff-pos-cap},
it provides a sufficient condition for the existence of nonconstant
\p-harmonic functions on $(X,d,\mu)$ with finite \p-energy.
Note that the Hausdorff dimension depends only on $\eps$,
$C_d$ and $R_0$, but not on $\be$ or $p$.

\begin{prop} \label{prop-Hausdorff-Xeps}
Assume that $\mu$ is doubling and
supports a \p-Poincar\'e inequality,
both properties holding for balls of radii at most $R_0$.
Let $\be > \be_0$.
Assume that the Borel set $E\subset \bdy_\eps X$
has positive $\kappa$-dimensional Hausdorff
measure for some $\kappa>(\log C_d)/\eps R_0$.
If $p=\be/\eps \ge 1$, then $\CpXeps(E)>0$.
\end{prop}


\begin{thebibliography}{99}

\bibitem{AikSh05} \art{\auth{Aikawa}{H} \AND \auth{Shanmugalingam}{N}}
        {Carleson-type estimates for \p-harmonic functions and the
        conformal Martin boundary of 
	John domains in metric measure spaces}
        {Michigan Math. J.}{53}{2005}{165--188}

\bibitem{BT} \art{Bishop, C. \AND Tyson, J. T.}
 {Locally minimal sets for conformal dimension}
 {Ann. Acad. Sci. Fenn. Math.}{26}{2001}{361--273}
 
\bibitem{ABkellogg} \art{Bj\"orn, A.}
         {A weak Kellogg property for quasiminimizers}
         {Comment. Math. Helv.} {81} {2006} {809--825}

\bibitem{ABremove} \art{Bj\"orn, A.}
        {Removable singularities for bounded \p-harmonic
          and quasi\-(super)\-harmonic functions on metric spaces}
        {Ann. Acad. Sci. Fenn. Math.} {31} {2006} {71--95}

\bibitem{BBbook} \book{\auth{Bj\"orn}{A} \AND \auth{Bj\"orn}{J}}
        {Nonlinear Potential Theory on Metric Spaces}
    {EMS Tracts in Mathematics {\bf 17},
        European Math. Soc., Z\"urich, 2011}

\bibitem{BBsemilocal} \art{\auth{Bj\"orn}{A} \AND \auth{Bj\"orn}{J}}	
        {Local and semilocal Poincar\'e inequalities on metric spaces}
        {J. Math. Pures Appl.} {119} {2018} {158--192}

\bibitem{BBnoncomp}  \art{\auth{Bj\"orn}{A} \AND \auth{Bj\"orn}{J}}	
         {Poincar\'e inequalities and Newtonian Sobolev functions on noncomplete 
         metric spaces}
         {J. Differential Equations.} {266} {2019} {44--69}
        
\bibitem{BBGS} \art{\auth{Bj\"orn}{A}, \auth{Bj\"orn}{J}, \auth{Gill}{J} 
        \AND \auth{Shan\-mu\-ga\-lin\-gam}{N}}
         {Geometric analysis on Cantor sets and trees}
         {J. Reine Angew. Math.}  {725} {2017} {63--114}

\bibitem{BBLintgreen} \artprep{\auth{Bj\"orn}{A}, \auth{Bj\"orn}{J}  \AND
         \auth{Lehrb\"ack}{J}}
         {Pointwise and integrability estimates for \p-harmonic
           Green functions in metric spaces}
         {\emph{In preparation}}
         
\bibitem{BBS} \art{Bj\"orn, A., Bj\"orn, J. \AND Shanmugalingam, N.}
        {The Dirichlet problem for \p-harmonic functions on metric spaces}
        {J. Reine Angew. Math.} {556} {2003} {173--203}

\bibitem{BBSpadm} \artprep{\auth{Bj\"orn}{A}, \auth{Bj\"orn}{J}  \AND
         \auth{Shanmugalingam}{N}}
         {Locally \p-admissible measures on $\R$}
        {\emph{Preprint}, 2018} {\tt arXiv:1807.02174}

\bibitem{BBSliouville} \artprep{\auth{Bj\"orn}{A}, \auth{Bj\"orn}{J}  \AND
         \auth{Shanmugalingam}{N}}
         {Liouville's theorem for \p-harmonic functions 
           and quasiminizers with finite energy}
         {\emph{Preprint}, 2018}
         {\tt arXiv:1809.07155}

\bibitem{BS-JMAA} \art{Bj\"orn, J. \AND Shanmugalingam, N.}
        {Poincar\'e inequalities, uniform domains and extension properties 
        for Newton--Sobolev functions in metric spaces}
        {J. Math. Anal. Appl.}{332}{2007}{190--208}

\bibitem{BHK-Unif} \book{Bonk, M., Heinonen, J. \AND Koskela, P.}
{Unifomizing Gromov Hyperbolic Spaces}
{Ast\'erisque {\bf 270} (2001)}

\bibitem{BS} \art{Bonk, M. \AND Saksman, E.}
   {Sobolev spaces and hyperbolic fillings}
   {J. Reine Angew. Math}
   {737}{2018} {161--187}
   
\bibitem{BSS} \art{\auth{Bonk}{M}, \auth{Saksman}{E} \AND \auth{Soto}{T}}
   {Triebel--Lizorkin spaces on metric spaces via hyperbolic fillings}
   {Indiana Univ. Math. J.}{67}{2018}{1625--1663} 

\bibitem{BridHaef} \book{Bridson, M \AND Haefliger, A.}
{Metric Spaces of Non-positive Curvature}
{Grundlehren der mathematischen Wissenschaften {\bf 319}, Springer, Berlin, 1999}

\bibitem{BHX} \art{\auth{Buckley}{S}, \auth{Herron}{D} \AND \auth{Xie}{X}}
  {Metric space inversions, quasihyperbolic distance, and uniform spaces}
  {Indiana Univ. Math. J.}{57}{2008}{837--890}

\bibitem{BuSch}\book{Buyalo, S. \AND Schroeder, V.}
  {Elements of Asymptotic Geometry}
 {EMS Monographs in Mathematics {\bf 3}, European Math. Soc., Z\"urich, 2007}

\bibitem{CDP} \book{Coornaert, M., Delzant, T. \AND Papadopoulos, A.}
{G\'eom\'etrie et th\'eorie des groupes}
{Lecture Notes in Math. {\bf 1441}, Springer, Berlin--Heidelberg, 1990}

\bibitem{DiM} \art{DiMarco, C.}
{Fractal curves and rugs of prescribed conformal dimension}
{Topology Appl.} {248} {2018} {117--127}

\bibitem{GeOs} \art{Gehring, F. W. \AND Osgood, B. G.}
{Uniform domains and the quasihyperbolic metric}
{J. Anal. Math.}{36}{1979}{50--74}

\bibitem{GdH} \artin{\auth{Ghys}{\'E} \AND \auth{de la Harpe}{P}}
    {Espaces m\'etriques hyperboliques}
    {\emph{Sur les groupes hyperboliques d'apr\`es Mikhael Gromov\/
      \textup{(}Bern, 1988\/\textup{)}}, 
     Progress in Math. {\bf 83}, pp. 27--45, Birkh\"auser Boston, Boston, MA, 1990}

  \bibitem{Gro} \artin{Gromov, M.}
  {Hyperbolic groups}
  {\emph{Essays in Group Theory},
    Math. Sci. Res. Inst. Publ. {\bf 8}, pp. 75--263, Springer, New York, 1987}

\bibitem{GromovBook} \book{\auth{Gromov}{M}}
  {Metric Structures for Riemannian and Non-Riemannian Spaces}
  {Progress in Math. {\bf  152}, Birkh\"auser Boston, Boston, MA, 1999}

\bibitem{HaKo} \book{\auth{Haj\l asz}{P} \AND \auth{Koskela}{P}}
	{Sobolev met Poincar\'e}
	{{Mem. Amer. Math. Soc.} {\bf 145}:688 (2000)}

\bibitem{heinonen} \book{Heinonen, J.}
        {Lectures on Analysis on Metric Spaces}
        {Springer, New York, 2001}

\bibitem{HeKiMa} \book{\auth{Heinonen}{J},
	\auth{Kilpel\"ainen}{T}
	\AND \auth{Martio}{O}}
        {Nonlinear Potential Theory of Degenerate Elliptic Equations}
        {2nd ed., Dover, Mineola, NY, 2006}

\bibitem{HeKo98} \art{\auth{Heinonen}{J} \AND \auth{Koskela}{P}}
	{Quasiconformal maps in metric spaces with controlled geometry}
	{Acta Math.} {181} {1998} {1--61}

\bibitem{HKST} \book{Heinonen, J., Koskela, P., Shanmugalingam, N. \AND Tyson, J. T.}
       {Sobolev Spaces on Metric Measure Spaces}
	{New Mathematical Monographs {\bf 27}, Cambridge Univ. Press,
        Cambridge, 2015}

\bibitem{HSX} \art{\auth{Herron}{D}, \auth{Shanmugalingam}{N} \AND \auth{Xie}{X}}
  {Uniformity from Gromov hyperbolicity}
  {Illinois J. Math.}{52}{2008}{1065--1109}

\bibitem{Keith} \art{Keith, S.}
        {Modulus and the Poincar\'e inequality on metric measure
         spaces} 
        {Math. Z.} {245} {2003} {255--292}

\bibitem{KiSh1} \art{Kinnunen, J. \AND Shanmugalingam, N.}
        {Regularity of quasi-minimizers on metric spaces}
        {Manuscripta Math.} {105} {2001} {401--423}

\bibitem{KoMc} \art{Koskela, P. \AND MacManus, P.}
        {Quasiconformal mappings and Sobolev spaces}
        {Studia Math.}{131}{1998}{1--17}

\bibitem{Lehr} \art{Lehrb\"ack, J.}
{Neighbourhood capacities}
{Ann. Acad. Sci. Fenn. Math.}{37}{2012}{35--51}

\bibitem{Sh-rev} \art{\auth{Shanmugalingam}{N}}
        {Newtonian spaces\textup{:} An extension of Sobolev spaces
        to metric measure spaces}
        {Rev. Mat. Iberoam.}{16}{2000}{243--279}

\bibitem{Sh-harm} \art{Shanmugalingam, N.} 
        {Harmonic functions on metric spaces}
        {Illinois J. Math.}{45}{2001}{1021--1050}

\end{thebibliography}
\end{document}